\documentclass[10pt]{amsart}
\usepackage{latexsym}
\usepackage{stmaryrd}
\usepackage{amsmath}
\usepackage{epsfig}
\usepackage{graphicx}
\usepackage{eufrak}
\usepackage{dsfont}
\usepackage[english]{babel}
\usepackage{color}
\usepackage{mathabx}

\usepackage{amscd}      
\usepackage{amssymb}
\usepackage{dsfont}
\usepackage{xypic}      
\LaTeXdiagrams          
\usepackage[all,v2]{xy}
\xyoption{2cell} \UseAllTwocells \xyoption{frame} \CompileMatrices
\usepackage{latexsym}
\usepackage{epsfig}
\usepackage{enumerate}

\usepackage{latexsym}
\usepackage{epsfig}
\usepackage{amsfonts}
\usepackage{enumerate}
\usepackage{mathrsfs}

\allowdisplaybreaks[3]

\newtheorem{prop}{Proposition}[section]
\newtheorem{lem}[prop]{Lemma}

\newtheorem{cor}[prop]{Corollary}
\newtheorem{thm}[prop]{Theorem}
\newtheorem{rmk}[prop]{Remark}

\newtheorem{defn}[prop]{Definition}

            {\nolinebreak $\Box$ \end{trivlist}}

\newcommand{\noprint}[1]{}

\renewcommand{\tilde}{\widetilde}

\newcommand{\E}{\mathop{\sf E}\nolimits}

\newcommand{\zz}{{\mathbb Z}}

\newcommand{\aaa}{{\mathbb A}}

\newcommand{\qq}{{\mathbb Q}}

\newcommand{\rr}{{\mathbb R}}

\newcommand{\Gm}{{{\mathbb G}_{\mbox{\tiny\rm m}}}}

\newcommand{\sD}{{\mathcal D}}
\newcommand{\sB}{{\mathcal B}}
\newcommand{\sC}{{\mathcal C}}

\newcommand{\sI}{{\mathcal I}}

\newcommand{\sG}{{\mathcal G}}
\newcommand{\sS}{{\mathcal S}}
\newcommand{\sN}{{\mathcal N}}

\newcommand{\sO}{{\mathcal O}}

\newcommand{\sX}{{\mathcal X}}
\newcommand{\sY}{{\mathcal Y}}
\newcommand{\sM}{{\mathcal M}}

\newcommand{\sZ}{{\mathcal Z}}

\newcommand{\sF}{{\mathcal F}}

\newcommand{\Coh}{\mbox{Coh}}

\newcommand{\SU}{\text{SU}}
\newcommand{\SL}{\text{SL}}
\newcommand{\PGL}{\text{PGL}}

\newcommand{\cEnd}{\mathscr{E}nd}
\newcommand{\cHom}{\mathscr{H}om}

\DeclareMathOperator{\Pic}{Pic}

\DeclareMathOperator{\Ch}{Ch}
\DeclareMathOperator{\CR}{CR}

\DeclareMathOperator{\Supp}{Supp}

\DeclareMathOperator{\Td}{Td}

\DeclareMathOperator{\End}{End}
\DeclareMathOperator{\MIC}{\textbf{MIC}}

\DeclareMathOperator{\can}{can}
\DeclareMathOperator{\Thm}{Thm}
\DeclareMathOperator{\pa}{pa}
\DeclareMathOperator{\Gal}{Gal}
\DeclareMathOperator{\Rat}{Rat}
\DeclareMathOperator{\padeg}{pa-deg}

\newcommand{\rk}{\mathop{\rm rk}}

\newcommand{\pr}{\mathop{\rm pr}\nolimits}
\newcommand{\Par}{\mathop{\rm Par}\nolimits}

\newcommand{\red}{\mathop{\rm red}\nolimits}

\newcommand{\spec}{\mathop{\rm Spec}\nolimits}

\numberwithin{equation}{subsection}

\newcommand {\mat}      [1] {\left(\begin{array}{#1}}
\newcommand {\rix}          {\end{array}\right)}

\def\Label{\label}

\title[Bogomolov-Gieseker inequality for tame Deligne-Mumford surfaces]{On the Bogomolov-Gieseker inequality for tame Deligne-Mumford surfaces}
\author{Yunfeng Jiang and  Promit Kundu}
\address{Department of Mathematics\\ University of Kansas\\ 405 Snow Hall 1460 Jayhawk Blvd\\Lawrence KS 66045 USA} 
\email{y.jiang@ku.edu}
\address{Department of Mathematics\\ University of Kansas\\ 405 Snow Hall 1460 Jayhawk Blvd\\Lawrence KS 66045 USA} 
\email{kundupromit63@ku.edu}

\address{Department of Mathematics, Shanghai Normal University, Shanghai 200234, People's Republic of China}
\email{hsun@shnu.edu.cn, hsunmath@gmail.com}

\begin{document}
\sloppy \maketitle
\begin{abstract}
We generalize the Bogomolov-Gieseker inequality for semistable coherent sheaves on smooth projective surfaces to smooth Deligne-Mumford surfaces.   We work over positive characteristic $p>0$ and generalize Langer's method to smooth Deligne-Mumford stacks.  As applications we obtain the  Bogomolov inequality for semistable coherent sheaves on a Deligne-Mumford surface in characteristic zero, and the Bogomolov inequality for semistable  sheaves on a root stack over a smooth surface which is equivalent to the Bogomolov inequality for the rational parabolic  sheaves on a smooth surface $S$.
In a joint appendix with Hao Max Sun, we generalize the Bogomolov inequality formula to Simpson Higgs sheaves on tame Deligne-Mumford stacks. 
\end{abstract}

\maketitle

\tableofcontents

\section{Introduction}

Let $X$ be a smooth projective surface over an algebraically closed field $\kappa$ of characteristic zero.  The Bogomolov-Gieseker inequality is a famous formula for slope semistable torsion free coherent sheaves on $X$, which says that the discriminant  $\Delta(E)=2\rk(E)c_2(E)-(\rk(E)-1)c_1(E)^2\geq 0$ if $E$ is slope semistable.  This formula has many important applications such as the construction of Bridgeland stability conditions for surfaces.  

If the base field $\kappa$ is of positive characteristic $p>0$, and $X$ is a smooth projective surface over $\kappa$. 
Let $F: X\to X$ be the absolute Frobenius morphism of $X$.  A torsion free coherent sheaf $E$ is called strongly slope semistable if any Frobenius pullback $F^*E$ is slope semistable.  For a strongly slope semistable torison free coherent sheaf $E$ on $X$, the Bogomolov-Gieseker inequality 
$$\Delta(E)=2\rk(E)c_2(E)-(\rk(E)-1)c_1(E)^2\geq 0$$ 
still holds, see \cite[Theorem 3.2]{Langer}.  In general if  $E$ is just slope semistable, the inequality has a correction term depending on the prime number $p$, see \cite[Theorem 3.3]{Langer}. The Bogomolov inequality formula in positive characteristics has applications to prove the boundedness for the moduli functor of semistable coherent sheaves on $X$, and the restriction theorem of slope stable torsion free coherent sheaves on $X$ to a divisor $D$ inside $X$.

In this paper we prove the  Bogomolov-Gieseker inequality  for slope semistable torsion free coherent sheaves on a smooth projective two dimensional Deligne-Mumford stack  $\sX$ (called a surface Deligne-Mumford stack)
in characteristic $p>0$. 
We work on tame surface Deligne-Mumford stacks which means that the orders of the local stabilizer groups of $\sX$ are all 
coprime  to the character $p$.
We use the modified slope semistability of Nironi \cite{Nironi} defined by generating sheaves $\Xi$ on $\sX$.  One motivation for our study on the  Bogomolov-Gieseker inequality for slope semistable torsion free coherent sheaves is the Vafa-Witten theory for projective surfaces and  surface Deligne-Mumford stacks in \cite{TT1}, \cite{JK}, where the Bogomolov-Gieseker inequality for the modified semistable sheaf $E$ will make the moduli space of Gieseker  stable sheaves on a root stack surface $\sX$ empty for $c_2(E)<0$.  The Vafa-Witten theory for surface Deligne-Mumford stacks has applications to prove the S-duality conjecture in \cite{VW} which is a functional duality for the generating functions counting 
$\SU(r)$ and $^{L}\SU(r)=\SU(r)/\zz_r$-instantons, see \cite{Jiang_2019}, \cite{Jiang_ICCM}, \cite{JKool}. 
On the other hand, the Bogomolov-Gieseker inequality  for slope semistable torsion free coherent sheaves on a surface Deligne-Mumford stack 
$\sX$ is interesting in its own since it will prove some restriction theorem of slope semistable sheaves on $\sX$ to a large degree divisor inside 
$\sX$.  This will have applications to the reduction  of the moduli of stable  Higgs bundles on surfaces to the moduli space of stable Higgs bundles on curves, which is related to the Langlands duality and mirror symmetry between the moduli spaces of $\SL_r$ and $\PGL_r$-Higgs bundles on curves. 

Let us first state our main result.  We fix $\sX$ to be  a surface Deligne-Mumford stack, and a polarization $(\Xi, \sO_X(1)=H)$ where $\Xi$ is a generating sheaf on $\sX$ and $\sO_X(1)$ a polarization on the coarse moduli space $X$.  We choose the generating sheaf $\Xi$ to satisfy the condition that its restriction to any codimension one component in $I\sX_1$ is a direct sum of locally free sheaves of the same rank. 
Here $I\sX_1\subset I\sX$ is the components in the inertia stack $I\sX$ consisting of codimension one components.  Then in this case the modified  slope of a torsion free coherent sheaf $E$ is given by:
$$\mu_{\Xi}(E)=\frac{\deg(E)}{\rk(\Xi)\rk(E)}$$
where $\deg(E)=c_1(E)\cdot H$.   The modified slope semistability of a torsion free sheaf $E$ is equivalent to the orbifold semistability of the surface Deligne-Mumford stack $\sX$ where the slope is given by $\mu(E)=\frac{\deg(E)}{\rk(E)}$.

We can not prove a Bogomolov-Gieseker inequality formula for a modified slope semistable torsion free sheaf $E$ for an arbitrary general generating sheaf 
$\Xi$, where the modified slope $\mu_{\Xi}(E)$ is $\frac{\rk(\Xi)\deg(E)}{\rk(E)}$ plus the contributions from codimension one components 
$I\sX_1\subset I\sX$.  
We do not know if such an inequality holds for general modified slope semistable torsion free sheaves, except that we can show it holds for root stacks and a special choice of generating sheaf. 
This has applications to the calculation of Vafa-Witten invariants for root stack surfaces in \cite{JK}.
For a root stack $\sX$ (associated to a pair $(X,D)$ where $D\subset X$ is a normal crossing divisor) over a smooth surface $X$, a choice of generating sheaf  will gives the equivalence between the category of coherent sheaves on $\sX$ and the category of rational parabolic sheaves on $(X,D)$. 

Our main result is:

\begin{thm}\label{thm_Deligne-Mumford_2-dim_intro}(Theorem \ref{thm_Deligne-Mumford_2-dim})
Let $\sX$ be a two dimensional smooth  tame projective Deligne-Mumford stack and let $\Xi$ be a generating sheaf 
such that its restriction to every component in $I\sX_1$ is a direct sum of locally free coherent sheaves of the same rank.  
Then if $E$ is a strongly modified slope semistable torsion free sheaf on $\sX$, we have 
$$\Delta(E)\geq 0.$$
\end{thm}

We prove Theorem \ref{thm_Deligne-Mumford_2-dim_intro} by a method of Moriwaki which is generalized to surface Deligne-Mumford stacks, plus the calculation using orbifold Grothendieck-Riemann-Roch formula for the surface Deligne-Mumford stack. 
In order to prove the theorem we review the basic knowledge of coherent sheaves over a tame Deligne-Mumford stack $\sX$, the Frobenius morphisms and the basic properties of slope modified semistable torsion free sheaves on  $\sX$.  In particular we generalize one of  Langer's inequality involving maximal  and minimal modified slopes in the Harder-Narasimhan filtration of a torsion free sheaf $E$ to smooth Deligne-Mumford stacks. 

As the first application of the  Bogomolov-Gieseker inequality in Theorem \ref{thm_Deligne-Mumford_2-dim_intro}, 
the Bogomolov-Gieseker inequality for surface Deligne-Mumford stacks in characteristic zero holds by the standard method of taking limit. 
The second application is to prove the Bogomolov-Gieseker inequality  for rational parabolic semistable torsion free sheaves on $(X,D)$ by relating the rational parabolic sheaves on  $(X,D)$ to torsion free sheaves on the root stack $\sX$ associated with $(X,D)$. 
The result in Theorem  \ref{thm_Deligne-Mumford_2-dim_intro}  can also be used to construct Bridgeland stability conditions on surface Deligne-Mumford stacks. 

We generalize Langer's results (the four theorems in \cite[\S 3]{Langer})  for higher dimensional smooth Deligne-Mumford stacks $\sX$ for a special choice of generating sheaf such that the modified slope is equivalent to the orbifold slope of $\sX$. 
We also provide one restriction theorem for slope stable torsion free sheaves to a large degree divisor in $\sX$. 
The proof is similar to \cite[\S 3]{Langer},  and we put these arguments in the appendix. 
The restriction theorem will have applications to the reduction of the moduli space of Higgs sheaves on a surface or a surface Deligne-Mumford stack  to the moduli space of Higgs  bundles to a large degree curve inside the surface or the surface Deligne-Mumford stack. 

We also prove the Bogomolov inequality for semistable  Simpson Higgs sheaves on tame Deligne-Mumford stacks in Appendix \ref{appendix_B} generalizing the method in \cite{Langer2}.  We should point out the Higgs sheaves here are not the Higgs sheaves in \cite{TT1}, \cite{JK}, where the Bogomolov inequality may involve correction terms.

\subsection{Outline}
Here is the short outline of the paper.  In \S \ref{sec_preliminaries} we review the modified slope stability and Frobenius morphisms for torsion free sheaves on a tame projective Deligne-Mumford stack $\sX$, and prove some inequalities for torsion free sheaves on $\sX$.  We prove the Bogomolov-Gieseker inequality formula Theorem \ref{thm_Deligne-Mumford_2-dim_intro} in \S \ref{sec_BG_formula_2-dim}, and the Bogomolov inequality for  rational parabolic sheaves for $(X,D)$.  Finally in Appendix A we generalize the results in \cite[\S 3, \S 5]{Langer} to higher dimensional smooth tame Deligne-Mumford stacks, and joint with Hao Sun we prove the Bogomolov inequality for Simpson Higgs sheaves on tame Deligne-Mumford stacks in Appendix  \ref{appendix_B}. 

\subsection{Convention}
We work over an algebraically closed   field $\kappa$ of characteristic $p>0$ throughout of the paper.  All Deligne-Mumford stacks are tame. We denote by $\Gm$ the multiplicative group over $\kappa$.   We use Roman letter $E$ to represent a coherent sheaf on a tame projective Deligne-Mumford stack  $\sX$.

\subsection*{Acknowledgments}

The authors   would like to thank  A. Langer for useful email correspondences and valuable discussions.  This work is partially supported by  NSF and a Simons Foundation.


\section{Preliminaries on modified stability and some new results}\label{sec_preliminaries}

In this section we review the modified semistability for Deligne-Mumford surfaces, and Langer's notations for semistability in characteristic $p$.
Some new results of Shapherd-Barron for Deligne-Mumford surfaces are proved. 

\subsection{Notations}\label{sbsec_notations}

We fix some notations for a smooth projective  Deligne-Mumford stack $\sX$ of dimension $d$.  The Deligne-Mumford stack $\sX$  is called $tame$, if  the stabilizer groups  of the Deligne-Mumford stack $\sX$ are all linearly reductive groups. Equivalently this means the order of the stabilizer group at any geometric point of $\pi: \sX\to X$ is relatively prime to $p$.

Let $\sI$ be the index set of the components of the inertia stack 
$I\sX$ such that
$$I\sX=\bigsqcup_{g\in\sI}\sX_g.$$
We always use $\sX_0=\sX$ to represent the trivial component.  For example if $\sX=[Z/G]$ is a global  quotient stack, where $Z$ is a quasi-projective scheme and $G$ is a finite group acting diagonally on $Z$, then 
$I\sX=\bigsqcup_{(g)}[Z^g/C(g)]$.   Any component $\sX_g\subset \sX$ in the inertia stack $I\sX$ is a closed substack of $\sX$. 
We denote by $I\sX_1\subset I\sX$ be the substack of $I\sX$ consisting of components $\sX_g$ such that their codimension in $\sX$ is one. 
Let $\pr: I\sX\to \sX$ be the map from the inertia stack $I\sX$ to $\sX$. 

For $\sX$, we write 
$$H_{\CR}^*(\sX)=H^*(I\sX)=\bigoplus_{g\in\sI}H^*(\sX_g)$$
to be the Chen-Ruan cohomology with $\qq$-coefficients.  For any torsion free coherent sheaf $E$ on $\sX$, we use $c_i(E)$ to represent the Chern classes of $E$ on $\sX$, and $c_i(E)\in H^{2i}(\sX)$. 

On the component $\sX_g\subset I\sX$, at a point $(x,g)\in \sX_g$, let 
$$T_{x}\sX=\bigoplus_{0\leq f<1}\left(T_{x}\sX\right)_{g,f}$$
be the eigenspace decomposition of $T_{x}\sX$ with respect to the stabilizer action and $g$ acts on $\left(T_{x}\sX\right)_{g,f}$ by $e^{2\pi i f}$.

Let $E\in\Coh(\sX)$ be a coherent sheaf on $\sX$, we have an eigenbundle decomposition of $\pr^*E$ and on $\pr^*E|_{\sX_g}$ we have
$$\pr^*E|_{\sX_g}=\bigoplus_{0\leq f<1}(\pr^*E)_{g,f}$$
with respect to the action of the stabilizer of $\sX_g$, where the element $g$ acts on $(\pr^*E)_{g,f}$ by $e^{2\pi i f}$.  
Then the orbifold Chern character is:
\begin{equation}\label{eqn_orbifold_Chern_E}
\widetilde{\Ch}(E)=\bigoplus_{g\in \sI}\sum_{0\leq f<1}e^{2\pi i f}\Ch((\pr^*E)_{g,f}),
\end{equation}
where $\Ch$ is the general Chern character.  Let $l_{g,f}$ be the rank of $(\pr^*E)_{g,f}$.   The orbifold Todd class of $T\sX$ is given by 
\begin{equation}\label{eqn_orbifold_Todd_TX}
\widetilde{\Td}(T\sX)=\bigoplus_{g\in \sI}\prod_{\substack{0\leq f<1\\
1\leq i\leq r_{g,f}}}\frac{1}{1-e^{-2\pi i f}e^{-x_{g,f,i}}}\prod_{f=0}\frac{x_{g,0,i}}{1-e^{-x_{g,0,i}}},
\end{equation}
where $(\pr^*T\sX)_{g,f}$ has rank $r_{g,f}$ and $x_{g,f,i}$ are Chern roots. 

For any coherent sheaf $E$ on $\sX$, orbifold Riemann-Roch theorem \cite{Toen} gives:
\begin{equation}\label{eqn_orbifold_RR}
\chi(\sX, E)=\int_{I\sX}\widetilde{\Ch}(E)\cdot \widetilde{\Td}(T\sX).
\end{equation}

\subsection{Modified stability}\label{subsec_modified_stability}

Let $\sX$ be a smooth tame projective Deligne-Mumford stack of dimension $d$.  We  choose the polarization $\sO_X(1)$ on its coarse moduli space $\pi: \sX\to X$.
Let $H:=c_1(\sO_X(1))$.
Recall from \cite{Nironi}, 

\begin{defn}\label{defn_generating_sheaf}
A locally free sheaf $\Xi$ on $\sS$ is $p$-very ample if  for every geometric point of $\sS$ the representation of the stabilizer group at that point contains every irreducible representation of the stabilizer group.  We call $\Xi$ a generating sheaf. 
\end{defn}

Let $\Xi$ be a locally free (generating) sheaf  on $\sX$. We define a functor
$$F_{\Xi}: D\Coh_{\sX}\to D\Coh_{X}$$ by
$$F\mapsto \pi_*\cHom_{\sO_X}(\Xi, F)$$
and a functor 
$$G_{\Xi}: D\Coh_{X}\to D\Coh_{\sX}$$ by
$$F\mapsto \pi^*F\otimes \Xi.$$
From \cite{OS03}, the functor $F_{\Xi}$ is exact since the dual $\Xi^{\vee}$ is locally free and the pushforward $p_*$ is exact.  The functor $G_{\Xi}$ is not exact unless $p$ is flat. For instance, if $p$ is a flat gerbe or a root stack, it is flat.

Let us fix a generating sheaf  $\Xi$ on $\sX$. We call the pair $(\Xi, \sO_S(1))$ a polarization of $\sX$. 
Let $E$ be a coherent sheaf on $\sX$, we define the support of $E$ to be the closed substack associated with the ideal
$$0\to \sI\to \sO_{\sX}\to \cEnd_{\sO_{\sX}}(F).$$
So $\dim(\Supp F)$ is the dimension of the substack associated with the ideal $\sI\subset \sO_{\sX}$ since $\sX$ is a Deligne-Mumford stack. 
A pure sheaf of dimension $d$ is a coherent sheaf $E$ such that for every non-zero subsheaf $E^\prime$ the support of $E^\prime$ is of pure dimension $d$. 
For any coherent sheaf $E$, we have the torsion filtration:
$$0\subset T_0(E)\subset \cdots \subset T_l(E)=E$$
where every $T_i(E)/T_{i-1}(E)$ is pure of dimension $i$ or zero, see \cite[\S 1.1.4]{HL}.

\begin{defn}\label{defn_Gieseker}
The modified Hilbert polynomial of a coherent sheaf $E$ on $\sX$ is defined as:
$$H_{\Xi}(E,m)=\chi(\sX, E\otimes\Xi^{\vee}\otimes \pi^*\sO_{X}(m))
=H(F_{\Xi}(E)(m))=\chi(X, F_{\Xi}(E)(m)).$$
\end{defn}

Let $E$ be of dimension $l$, then we can write:
$$H_{\Xi}(E,m)=\sum_{i=0}^{l}\alpha_{\Xi, i}(E)\frac{m^i}{i!}$$
which is induced by the case of schemes.  The modified Hilbert polynomial is additive on short exact sequences since the functor $F_{\Xi}$ is exact.  If we don't choose the generating sheaf $\Xi$, the Hilbert polynomial $H$ on $\sX$ will be the same as the Hilbert polynomial on the coarse moduli space $S$.  
The {\em reduced modified Hilbert polynomial} for the pure sheaf $E$ is defined as 
$$h_{\Xi}(E)=\frac{H_{\Xi}(E)}{\alpha_{\Xi, d}(E)}.$$
Let $E$ be a pure coherent sheaf.  We call $E$ Gieseker  semistable if for every proper subsheaf $E^\prime\subset E$, 
$$h_{\Xi}(E^\prime)\leq h_{\Xi}(E).$$
We call $E$ stable if $\leq$ is replaced by $<$ in the above inequality.


\begin{defn}(\cite[Definition 3.13]{Nironi})\label{defn_modified_slope_stability}
 We define the slope of  $E$ by 
 $$\mu_{\Xi}(E)=\frac{\alpha_{\Xi, l-1}(E)}{\alpha_{\Xi, l}(E)}.$$
 Then $E$ is
modified slope (semi)stable if for every proper subsheaf $F\subset E$, 
$$\mu_{\Xi}(F)(\leq) < \mu_{\Xi}(E).$$
\end{defn}
The notion of $\mu$-stability and semistability is related to the Gieseker stability and semistability in the same way as schemes, i.e.,
$$\mu-\text{stable}\Rightarrow \text{Gieseker stable}\Rightarrow \text{Gieseker semistable}\Rightarrow \mu-\text{semistable}$$

\begin{rmk}\label{rmk_stability_quotient}
Recall that  $I\sX_1\subset I\sX$ is the substack of $I\sX$ consisting of components such that the codimension of $\sX_g\subset \sX$  is one. 
If our Deligne-Mumford stack $\sX$ is a global quotient stack, which means $\sX=[Z/G]$ where $Z$ is a quasi-projective scheme and $G$ is a group scheme acting diagonally.   Assume that we can choose the generating sheaf $\Xi$ on $\sX$ such that its restriction on any component in $I\sX_1$ is a sum of  locally free sheaves of  the same rank. If the sheaf $E$ has dimension $d$, then    \cite[Proposition 3.18]{Nironi} shows that 
$$\deg_{\Xi}(E)=\frac{1}{\rk(\Xi)}\alpha_{\Xi, d-1}(E)-\frac{\rk(E)}{\rk(\Xi)}\alpha_{\Xi,d-1}(\sO_{\sX}).$$
Here $\alpha_{\Xi, d-1}(E)=\int_{\sX}c_1(E)\cdot H$. 
\end{rmk}

\subsection{Notations in characteristic $p$}\label{subsec_char_p_notations_Shepherd-Barron}

We first fix some notations in characteristic $p>0$.  We assume $\sX$ a smooth tame  projective Deligne-Mumford stack of dimension $d$.  Let 
$$\sX^{(i)}=\sX\times_{\spec \kappa}\spec(\kappa)$$
be the Deligne-Mumford stack obtained from  $\sX$ by applying the $i$-th power of absolute Frobenius morphism on $\spec \kappa$. 
The geometric Frobenius morphism $F_g: \sX\to \sX^{(1)}$ is defined by the fiber product $\sX^{(1)}=\sX\times_{\spec \kappa}\spec(\kappa)$ and the absolute Frobenius morphism $F: \sX\to \sX$.

We review a bit on the coherent sheaves with connections in \cite{Katz}, where the theory is for schemes, but in \'etale topology it works for Deligne-Mumford stacks. 
Following N.Katz \cite{Katz}, a connection on a quasi-coherent sheaf $E$ is a homomorphism 
$$\nabla: E\to E\otimes \Omega_{\sX}$$
such that $\nabla(ge)=g\nabla(e)+e\otimes dg$ where $g$ and $e$ are sections of $\sO_{\sX}$ and $E$ respectively over an open subset of $\sX$, and $dg$ denotes the image of $g$ under the canonical exterior differentiation $d: \sO_{\sX}\to \Omega_{\sX}$. 
The kernel $K:=K(E,\nabla)$ of $\nabla$ is the $\sO_{\sX}$-linear map
$K=\nabla_1\circ\nabla: E\to E\otimes \Omega_{\sX}^{2}$ where $\nabla_1: E\otimes \Omega_{\sX}\to E\otimes \Omega_{\sX}^{2}$ is the map defined by
$\nabla_1(e\otimes \omega)=e\otimes d\omega-\nabla(e)\wedge \omega$. 
A connection $\nabla: E\to E\otimes \Omega_{\sX}$ is $integrable$ if its kernel 
$K=0$. 
Then let $\MIC(\sX)$ be the category of pairs $(E,\nabla)$ where $E$ is a quasi-coherent $\sO_{\sX}$-module and $\nabla$ is an integrable connection. 

The $p$-curvature of an integrable connection  $\nabla$ is given by  a morphism 
$$\psi: \mbox{Der}_{\kappa}(\sX)\to \End(E),$$
by 
$$\psi(D)=(\nabla(D))^{p}-\nabla(D^p),$$
where $D$ is a section of $\mbox{Der}_{\kappa}(\sX)$. 

For the geometric Frobenius morphism $F_g: \sX\to \sX^{(1)}$, 
in \cite[Theorem 5.1]{Katz},  N. Katz constructed a canonical connection $\nabla_{\can}$ on $F_g^*E$ which is determined in the following Cartier theorem:

There is an equivalence of categories between the category of quasi-coherent sheaves on $\sX^{(1)}$ and the full subcategory of  $\MIC(\sX)$ consisting of objects $(E,\nabla)$ whose $p$-curvature is zero. The equivalence is given by:
$E\mapsto (F_g^*E, \nabla_{\can})$ and $(E,\nabla)\mapsto  E^{\nabla}$. Here the unique $\nabla_{\can}$ on $F_g^*E$ makes 
$E\cong (F_g^* E)^{\nabla_{\can}}$, and for a sheaf $E$, $E^{\nabla}$ is the kernel $\ker(\nabla)$ of $\nabla$.

We first have a result of generating sheaves under Frobenius pullbacks:

\begin{lem}
Let $F_g: \sX\to \sX^{(1)}$ be the geometric Frobenius morphism. Then a locally free sheaf  $\Xi$ is a generating sheaf on $\sX^{(1)}$ if and only if $F_g^*\Xi$ is a generating sheaf on $\sX$.
\end{lem}
\begin{proof}
That $\Xi$ is a generating sheaf  means it contains all the irreducible representations of the stabilizer group of $\sX^{(1)}$.  Then the pullback 
$F_g^*\Xi$ contains all the irreducible representations of the stabilizer group of $\sX$ since $p$ is prime to all the orders of the stabilizer groups of $\sX$. 
\end{proof}

Thus from the Cartier's theorem above, the following is a generalization of \cite[Proposition 2.2]{Langer}:
\begin{prop}\label{prop_geometric_Frobenius_semistability}
A coherent sheaf $E$ on $\sX^{(1)}$ is slope semistable with respect to $(H, \Xi)$ if and only of $F_g^*E$ is slope semistable with respect to 
$(F_g^{*}H, F_g^*\Xi)$.
\end{prop}
\begin{proof}
This is from the modified slope of $E$ is the product of some $p$-th power with the modified  slope of $F_g^*E$.
\end{proof}

The following lemma is the generalization of \cite[Lemma 2.3]{Langer}.
We recall that a sheaf  $E$ is $\nabla$-semistable if the inequality of modified slopes is satisfied for all nonzero 
$\nabla$-preserved subsheaves of $E$. 
\begin{lem}\label{lem_E_nabla}
Consider a torsion free sheaf $E$ on $\sX$ with a $\kappa$-connection $\nabla$,  and assume that $E$ is $\nabla$-semistable. 
Let $0=E_0\subset E_1\subset\cdots\subset E_m=E$ be the usual Harder-Narasimhan filtration, then the induced morphisms
$E_i\to (E/E_i)\otimes \Omega_{\sX}$ are nonzero $\sO_{\sX}$-morphisms. 
\end{lem}
\begin{proof}
If the induced morphisms
$E_i\to (E/E_i)\otimes \Omega_{\sX}$ are zero, then it induces a  Harder-Narasimhan filtration for $\nabla$-connections which contradicts the 
$\nabla$-semistability. 
\end{proof}

Now we fix some notations for the slopes of a coherent sheaf $E$ on $\sX$ with respect to the polarization $(\Xi, H)$.   Let 
$$0=E_0\subset E_1\subset \cdots \subset E_m:=E$$
be the Harder-Narasimhan filtration of $E$ with respect to modified slope stability.  We denote by 
$\mu_{\Xi, \max}(E), \mu_{\Xi, \min}(E)$ the maximal and minimal modified slope of $E$ respectively.  We let
$$L_{\Xi,\max}(E)=\lim_{k\to \infty}\frac{\mu_{\Xi,\max}((F^k)^*E)}{p^k};\quad  
L_{\Xi,\min}(E)=\lim_{k\to \infty}\frac{\mu_{\Xi,\min}((F^k)^*E)}{p^k}.$$

The sequence $\frac{\mu_{\Xi,\max}((F^k)^*E)}{p^k}$ is weakly increasing and $\frac{\mu_{\Xi,\min}((F^k)^*E)}{p^k}$ is weakly decreasing. 
Also $L_{\Xi,\max}(E)\geq \mu_{\max}(E)$ and $L_{\Xi,\min}(E)\leq \mu_{\min}(E)$. 
Suppose that $E$ is slope semistable, then $L_{\Xi,\max}(E)= \mu_{\Xi}(E)$ (or $L_{\Xi,\min}(E)= \mu_{\Xi}(E)$) if and only if $E$ is strongly slope semistable. 

First we have:
\begin{lem}\label{lem_SB1}
Let $E$ be a slope semistable torsion-free sheaf on $\sX$ such that $F^*E$ is unstable.  Then in the  Harder-Narasimhan filtration
$0=E_0\subset E_1\subset \cdots \subset E_m:=F^*E$, the $\sO_{\sX}$-morphisms
$E_i\to (F^*E/E_i)\otimes \Omega_{\sX}$ induced by $\nabla_{\can}$ are nonzero. 
\end{lem}
\begin{proof}
Note that this is Proposition 1 in \cite{SB1} where Shepherd-Barron dealt with general slope stability for schemes. 
The result for modified slope stability  is from Proposition \ref{prop_geometric_Frobenius_semistability} and Lemma \ref{lem_E_nabla}.
\end{proof}

For any divisors $H, A$, $H^i\cdot A^{d-i}$ denotes the integration $\int_{\sX}H^i\cdot A^{d-i}$. Here 
We still denote $H$ and $A$ for $\pi^*H, \pi^*A$ where $\pi: \sX\to X$ is the map to its coarse moduli space. 
\begin{cor}\label{cor_alphaE}
Let us define
$$\alpha_{\Xi}(E):=\max(L_{\Xi,\max}(E)-\mu_{\Xi, \max}(E), \mu_{\Xi, \min}(E)-L_{\Xi,\min}(E)).$$
Let $L\in \Pic(\sX)$ be a $\pi$-ample line bundle on $\sX$ such that there exists $m\in\zz_{>0}$ making 
$\bigoplus_{i=1}^{m-1}L^{\otimes i}$ a generating sheaf.    
Let $A$ be a nef divisor for $X$,  such that $\pi_{*}(T_{X}\otimes\sum_{k=1}^{m-1}(L^{k}))(A)$ is globally generated, 
we have 
$$\alpha_{\Xi}(E)\leq \frac{\rk(E)-1}{p-1} \left(\max_{1\leq k\leq m-1}(\mu_{\Xi}(A\otimes L^{k}))\right),$$
where 
$$\max_{1\leq k\leq m-1}(\mu_{\Xi}(A\otimes L^{k}))=\max\{\mu_{\Xi}(A\otimes L), \mu_{\Xi}(A\otimes L^{2}), \cdots, \mu_{\Xi}(A\otimes L^{m-1})\}.$$
\end{cor}
\begin{proof}
This is similar to  \cite[Cor 2.5]{Langer}.
We modify our case by observing that $\sum_{k=1}^{m-1}L^{k}$ is a generating sheaf on $\sX.$ So we have 
$$\pi^{*}(\pi_{*}(T\sX\otimes\sum_{k=1}^{m-1}L^{k}))\otimes\sum_{k=1}^{m-1}L^{-k}\to T\sX$$
which is a surjection.
From our assumption $\pi_{*}(T\sX\otimes\sum_{k=1}^{m-1}L^{k})(A)$ is globally generated. Using the right exactness of pullback we obtain,
$$((\sO_{\sX}(-A))^{r}\otimes\sum_{k=1}^{m-1}L^{-k})\to \pi^{*}\pi_{*}(T\sX\otimes\sum_{k=1}^{m-1}L^{k})\otimes \sum_{k=1}^{m-1}L^{-k}\to T\sX$$ which is again a surjection.\\
Dualising we obtain,
$$0\to \Omega_{\sX}\to ((\sO_{\sX}(A))^{r}\otimes\sum_{k=1}^{m-1}L^{k})$$

Recall the  Harder-Narasimhan filtration
$0=E_0\subset E_1\subset \cdots \subset E_m:=F^*E$ for $F^*E$, 
tensoring by $F^{*}(E)/E_{i}$ and taking $\mu_{\Xi,max}$ we obtain,
$$\mu_{\Xi,\max}((F^{*}(E)/E_{i})\otimes \Omega_{\sX})\leq \mu_{\Xi,\max}(((F^{*}(E)/E_{i})\otimes (\sO_{\sX}(A))^{r}\otimes\sum_{k=1}^{m-1}L^{k}))$$
We now use $$\mu_{\Xi,\max}(E\oplus F)=max(\mu_{\Xi,\max}(E),\mu_{\Xi,\max}(F))$$
where $E,F$ are torsion free coherent sheaves.\\
Using the above the right hand side reduces to the following,
$$\max_{1\leq k\leq m-1}(\mu_{\Xi,\max}(F^{*}(E)/E_{i}\otimes A\otimes L^{k})).$$
Now we estimate,
$$\mu_{\Xi,\max}(F^{*}E/E_{i}\otimes A\otimes L^{k})=\mu_{\Xi}(E_{i+1}/E_{i})+\mu_{\Xi}(A\otimes L^{k}).$$
Taking max into account we obtain,
$$\max_{1\leq k\leq m-1}(\mu_{\Xi,\max}(F^{*}(E)/E_{i}\otimes A\otimes L^{k}))=\max_{1\leq k\leq m-1}(\mu_{\Xi}(A\otimes L^{k}))+\mu_{\Xi}(E_{i+1}/E_{i}).$$
Also from Lemma \ref{lem_SB1}, $\mu_{\Xi, \min}(E_i)\leq \mu_{\Xi, \max}((F^*E/E_i)\otimes \Omega_{\sX})$.
Thus from above we get: 
$$\mu_{\Xi}(E_{i}/E_{i-1})-\mu_{\Xi}(E_{i+1}/E_{i})\leq \max_{1\leq k\leq m-1}(\mu_{\Xi}(A\otimes L^{k}))$$
for each $i=1, \cdots, m-1$. 
Thus summing over above inequalities we obtain
$$\mu_{\Xi,\max}(F^{*}(E))-\mu_{\Xi,\min}(F^*E)\leq (\rk(E)-1)\max_{1\leq k\leq m-1}(\mu_{\Xi}(A\otimes L^{k})).$$

Note that $\frac{\mu_{\Xi, \max}((F^k)^*E)}{p^k}$ is weakly increasing, and 
$\mu_{\Xi,\max}(E)\geq \frac{\mu_{\Xi,\min}(F^*E)}{p}$, we have by induction that 
$$\frac{\mu_{\Xi, \max}((F^k)^*E)}{p^k}-\mu_{\Xi,\max}(E)\leq \frac{\rk(E)-1}{p-1} \max_{1\leq k\leq m-1}(\mu_{\Xi}(A\otimes L^{k}))$$
Taking limit we get the inequality for $L_{\Xi,\max}(E)-\mu_{\Xi, \max}(E)$, and the corresponding inequality for $\mu_{\Xi, \min}(E)-L_{\Xi,\min}(E))$ can be similarly proved.  
\end{proof}

\subsection{fdHN property}\label{subsec_fdHN}

As in \cite[\S 2.6]{Langer}, a torsion free coherent sheaf $E$ on $\sX$ has an fdHN property (finite determinacy of the Harder-Narasimhan filtration) if there exists a $k_0$ such that all quotients in the Harder-Narasimhan filtration of $(F^{k_0})^*E$ are strongly modified slope semistable with respect to $\Xi$. 
$E$ is fdHN of it has an fdHN property.   
Similar to the proof in \cite[Theorem 2.7]{Langer}, we have

\begin{prop}\label{prop_fdHN}
Let $E$ be a torsion free sheaf on $\sX$, the $E$ is fdHN. 
\end{prop}
\begin{proof}
Let 
$$0=E_0\subset E_1\subset\cdots\subset E_m=E$$
be the Harder-Narasimhan filtration of $E$.  We have the HNP (Harder-Narasimhan polygon) associated with $E$ defined by connecting the points $p(E_0), p(E_1), \cdots, p(E_m)$ by successively line segments and connecting the last one with the first one. Here 
$p(G)=(\alpha_{\Xi, d-1}(G), \alpha_{\Xi, d}(G))$. 
Since the base field $\kappa$ has characteristic $p>0$, there is a sequence of polygons $HNP_k(E)$, where $HNP_k(E)$ is defined by  contracting $HPN((F^k)^*E)$ along the $\alpha_{\Xi, d-1}$ axis by the factor $1/p^k$. 
All of the polygons are bounded and $HNP_k(E)$ is contained in $HNP_{k+1}(E)$. 
Then the proof of \cite[Theorem 2.7]{Langer} goes through without major changes.  We omit the details. 
\end{proof}


\section{Bogomolov-Gieseker  inequality}\label{sec_BG_formula_2-dim}

In this section 
we always fix $\sX$ as a smooth tame projective Deligne-Mumford stack of dimension $2$. Interesting two dimensional Deligne-Mumford stacks are reviewed in \cite{JK}, where the name surface Deligne-Mumford stack is used.  We always fix a generating sheaf $\Xi$ on $\sX$ and talk about modified semi-stable sheaves in characteristic p and assume the existence of a $\pi$ ample line bundle on $\sX$. 

\subsection{Condition $\star$}\label{condition_star}
We say a  generating sheaf $\Xi$ on $\sX$  satisfying Condition $\star$ if   its restriction on any component in $I\sX_1$ is a sum of  locally free sheaves of  the same rank  as in    \cite[Proposition 3.18]{Nironi} and Remark \ref{rmk_stability_quotient}. 

We have
$$\deg_{\Xi}(E)=\frac{1}{\rk(\Xi)}\alpha_{\Xi, d-1}(E)-\frac{\rk(E)}{\rk(\Xi)}\alpha_{\Xi,d-1}(\sO_{\sX}).$$
Here $\deg(E)=\int_{\sX}c_1(E)\cdot H^{d-1}=\alpha_{\Xi, d-1}(E)$.
Also $\alpha_{\Xi, d}(E)=\rk(E)\rk(\Xi)H^d$. Since $\frac{\rk(E)}{\rk(\Xi)}\alpha_{\Xi,d-1}(\sO_{\sX})$ is constant for a fixed polarization 
$\sO_X(1)=H, \Xi$, it is reasonable to use 
$\mu_{\Xi}(E)=\frac{\rk(\Xi)\deg_{\Xi}(E)}{\rk(E)}$ as the definition of modified slope.  Note that in this case the modified slope stability is the same as orbifold stability 
defined by $\mu(E)=\frac{\deg(E)}{\rk(E)}$.

For any torsion free coherent sheaf $E$ on $\sX$ of rank $\rk$, recall the Chern character morphism in (\ref{eqn_orbifold_Chern_E}), we rewrite it here
$$\widetilde{\Ch}(E)=\bigoplus_{g\in \sI}\sum_{0\leq f<1}e^{2\pi i f}\Ch((\pr^*E)_{g,f}).$$
We set:
$$\Delta(E):=\Ch_1(E)^2-2\Ch_0(E)\cdot \Ch_2(E)=2\rk(E)c_2(E)-(\rk(E)-1)c_1^2(E).$$

\subsection{Two dimensional case}
  
In this section we prove the following  result:
\begin{thm}\label{thm_Deligne-Mumford_2-dim}
Let $\sX$ be a two dimensional smooth projective Deligne-Mumford stack.  
Then if $E$ be a strongly modified slope semistable torsion free sheaf on $\sX$, we have 
$$\Delta(E)\geq 0.$$
\end{thm}

We use the method of \cite{Moriwaki}.  Let us first prove a lemma.

\begin{lem}\label{lem_thm_2dim_1}
Suppose that $E$ is strongly modified slope semistable torsion free coherent sheaf of rank $\rk(E)$ on $\sX$ and $c_1(E)\cdot H=0$. Let $L$ be a line bundle on $\sX$, then there exists a positive number $M$ such that 
$$h^0(\sX, (F^n)^*(E)\otimes L)\leq M\cdot p^n$$
for sufficiently large integers $n$. 
\end{lem}
\begin{proof}
First we take a positive integer $m$ such that $L\cdot H-mH^2<0$.  Therefore we have 
$c_1((F^n)^*(E)\otimes L\otimes H^{-m})<0$.  Since $(F^n)^*(E)\otimes L\otimes H^{-m}$ is modified slope semistable,  we have 
$H^0(\sX, (F^n)^*(E)\otimes L\otimes H^{-m})=0$ (otherwise it will  contradicts with the modified semistability).  
Then we choose a general element $C\in |mH|$ and consider $\sC=\pi^{-1}(C)$, and the exact sequence
$$0\to \sO_{\sX}(-\sC)\longrightarrow \sO_{\sX}\longrightarrow \sO_{\sC}\to 0.$$
Tensor with $(F^n)^*(E)\otimes L$ we get an inequality 
$$h^0(\sX, (F^n)^*(E)\otimes L)\leq h^0(\sC, (F_{\sC}^n)^*(E|_{\sC})\otimes L|_{\sC})$$
where $F_{\sC}$ is the absolute Frobenius morphism of $\sC$, and $E|_{\sC},  L|_{\sC}$ are the restrictions to $\sC$. 

We show that there exists a positive number $M$ such that 
$$h^0(\sC, (F_{\sC}^n)^*(E|_{\sC})\otimes L|_{\sC})\leq M\cdot p^n.$$

We prove it for any vector bundle $E$ on $\sC$ of rank $\rk(E)$ and a line bundle $L$ on $\sC$. 
The rank one case of $E$ is obvious since $(F^n)^*E=E^{\otimes p^n}$ and by  orbifold Riemann-Roch theorem \cite{Toen}.  
General case is proved by induction on the rank $\rk(E)$.  Consider an exact sequence
$$0\to G\to E\to Q\to 0$$
where $G$ is a line bundle and $Q$ is a rank $\rk-1$ vector bundle, and the exact sequence:
$$0\to (F^n)^*G\otimes L\longrightarrow (F^n)^*E\otimes L\longrightarrow (F^n)^*Q\otimes L\to 0.$$
Thus we have
$$h^0(\sC, (F_{\sC}^n)^*(E)\otimes L)\leq h^0(\sC, (F_{\sC}^n)^*(G)\otimes L)+h^0(\sC, (F_{\sC}^n)^*(Q)\otimes L).$$
We are done.
\end{proof}

\subsection*{Proof of Theorem \ref{thm_Deligne-Mumford_2-dim}}

We claim that we can assume $c_1(E)=0$. Let $\rk:=l\cdot p^j$ for some integers $l$ and $j$, then $l$ is prime to $p$. Then from \cite[Lemma 2.1]{BG}, since $X$ is a scheme, there is a separable finite morphism $\phi: Y\to X$ such that $c_1(\phi^*E)$ is divisible by $l$.  
Note that \cite[Lemma 2.1]{BG} proved the statement for smooth schemes, but since we are in two dimensional, the argument works for surfaces with at most of quotient singularities.  The coarse moduli space $X$ always has quotient singularities. 
Let 
$f:=\phi\circ F^j$, then $c_1(f^*E)$ is divisible by $\rk$. 
We form the cartesian diagram:
\begin{equation}\label{eqn_key_diagram_Galois}
\xymatrix{
\sY\ar[r]^{\tilde{f}}\ar[d]_{\pi_{Y}}& \sX\ar[d]^{\pi}\\
Y\ar[r]^{f}& X
}
\end{equation}
where $\sY$ is the fibre product.  Then $\sY$ is a Deligne-Mumford stack with the stacky locus all pullback from the stacky locus of $\sX$. 
We set:
$$\widetilde{E}:=\tilde{f}^*(E)\otimes \sO_{\sY}\left(-\frac{c_1(\tilde{f}^*(E))}{\rk}\right).$$
Then we calculate
$$c_1(\widetilde{E})=0; \quad  c_2(\widetilde{E})=\tilde{f}^*\left(c_2(E)-\frac{\rk-1}{2\rk}c_1(E)^2\right).$$
Since $E$ is strongly modified slope semistable with the slope given by $\mu_{\Xi}(E)=\frac{\deg(E)}{\rk(E)\cdot \rk(\Xi)}$,  from \cite[Lemma 1.1]{Gieseker}, $\tilde{f}^*E$ is modified slope semistable.  Therefore $c_2(\widetilde{E})\geq 0$ implies that 
$\Delta(E)\geq 0$.

Then we assume $c_1(E)=0$.  By Lemma \ref{lem_thm_2dim_1}, there exists positive numbers $M_1, M_2$ such that 
$$h^0(\sX, (F^n)^*(E))\leq M_1\cdot p^n; \quad h^0(\sX, (F^n)^*(E^{\vee})\otimes \omega_{\sX})\leq M_2\cdot p^n$$
for sufficiently large integers $n$.   
In the inertia stack $I\sX$, $\sI_1$ denotes the index set such that $\sX_g\subset \sX$ has codimension one for $g\in\sI_1$. We let $\sI_2$ denotes the index set such that $\sX_g\subset \sX$ has codimension two, i.e., the stacky locus consisting of points in $\sX$. 
Then by orbifold Riemann-Roch theorem (\ref{eqn_orbifold_RR}) from \cite{Toen}, 
\begin{align*}
\chi(\sX,  (F^n)^*(E))&=\int_{I\sX}\widetilde{\Ch}( (F^n)^*(E))\cdot \widetilde{\Td}(T_{\sX})\\
&=\int_{\sX}\Ch((F^n)^*(E))\cdot \Td(T_{\sX})+\\
&+\sum_{g\in\sI_1}\int_{\sX_g}\sum_{0\leq f<1}e^{2\pi i f}\Ch( (F^n)^*(E)_{g,f})\cdot \Td(T_{\sX})_{g,f}\\
&+ \sum_{g\in\sI_2}\int_{\sX_g}\sum_{0\leq f<1}e^{2\pi i f}\Ch( (F^n)^*(E)_{g,f})\cdot \Td(T_{\sX})_{g,f}\\
&=-c_2((F^n)^*(E))+\int_{\sX}\rk\cdot \Td(T_{\sX})\\
&+\sum_{g\in\sI_1}\left(\int_{\sX_g}\sum_{0\leq f<1}e^{2\pi i f}c_1((F^n)^*(E)_{g,f})+\int_{\sX_g}\sum_{0\leq f<1}e^{2\pi i f}\cdot \Td(T_{\sX})_{g,f}\right)\\
&+ \sum_{g\in\sI_2}\int_{\sX_g}\sum_{0\leq f<1}e^{2\pi i f}\cdot \Td(T_{\sX})^{0}_{g,f}
\end{align*}
where $\Td(T_{\sX})^{0}_{g,f}$ is the constant term of $\Td(T_{\sX})_{g,f}$. 
By the Frobenius pullback property of Chern classes we have 
\begin{align*}
\chi(\sX,  (F^n)^*(E))
&=-p^{2n}c_2(E)+\int_{\sX}\rk\cdot \Td(T_{\sX})\\
&+\sum_{g\in\sI_1}\left(p^{n}\int_{\sX_g}\sum_{0\leq f<1}e^{2\pi i f}c_1((E)_{g,f})+\int_{\sX_g}\left(\sum_{0\leq f<1}e^{2\pi i f}\right)\cdot \Td(T_{\sX})_{g,f}\right)\\
&+ \sum_{g\in\sI_2}\int_{\sX_g}\left(\sum_{0\leq f<1}e^{2\pi i f}\right)\cdot \Td(T_{\sX})^{0}_{g,f}\\
&\leq (M_1+M_2)p^n
\end{align*}
Hence for large $n$, to ensure the above inequality  we must have $c_2(E)\geq 0$.  We are done. 
$\square$.

\begin{rmk}
If the base field $\kappa$ has character zero, and $E$ is a modified semistable  torsion free sheaf on $\sX$, then  the Bogomolov inequality $\Delta(E)\geq 0$  also holds by the standard method of taking limit, see \cite{Langer}.   We omit the details here, and note that in  the character zero case Bogomolov inequality for orbifold semistable torsion free sheaves is proved in \cite[Lemma 2.5]{Kawamata} for surface Deligne-Mumford stacks with only isolated quotient singularities. 
\end{rmk}

\subsection{Bogomolov inequality for parabolic sheaves}\label{subsec_parabolic_sheaves}

In this section we give an application of the Bogomolov inequality in Theorem \ref{thm_Deligne-Mumford_2-dim} to rational parabolic sheaves on a surface
$X$. 

\subsubsection{Root surfaces}

Let $X$ be a smooth projective surface and $D\subset X$ is a simple normal crossing Cartier divisor.  Let $r\in \zz_{>0}$ be a positive integer.   
The line bundle with the section $(\sO_X(D), s_D)$ defines a morphism 
$$S\to  [\aaa^1/\Gm].$$
Let $\Theta_r:  [\aaa^1/\Gm]\to  [\aaa^1/\Gm]$ be the morphism of stacks given by the morphism 
$$x\in \aaa^1\mapsto x^r\in \aaa^1; \quad  t\in \Gm\mapsto t^r\in \Gm,$$
which sends $(\sO_X(D), s_D)$ to $(\sO_X(D)^{\otimes r}, s^r_D)$.
\begin{defn}\label{defn_root_stack}
Let $\sX:=\sqrt[r]{(X,D)}$ be the stack obtained by the fibre product 
\[
\xymatrix{
\sqrt[r]{(X,D)}\ar[r]\ar[d]_{\pi}& [\aaa^1/\Gm]\ar[d]^{\Theta_r}\\
X\ar[r]^--{(\sO_X(D), s_D)}& [\aaa^1/\Gm].
}
\]
We call $\sX=\sqrt[r]{(X,D)}$ the root stack obtained from $X$ by the $r$-th root construction. 
\end{defn}

The stack $\sX=\sqrt[r]{(X,D)}$ is a smooth Deligne-Mumford stack with stacky locus $\sD:=\pi^{-1}(D)$, and $\sD\to D$ is a $\mu_r$-gerbe over $X$ coming from the line bundle $\sO_S(D)|_{D}$. 

\begin{rmk}
The general root stacks over a logarithmic scheme $X$ is constructed in \cite{BV}, and the pair $(X,D)$ defines a canonical log structure on $X$.  
Since we don't need the abstract language of log schemes we refer the reader to \cite{BV} for details. 
\end{rmk}
 
Theorem \ref{thm_Deligne-Mumford_2-dim} implies the following result:
\begin{prop}\label{prop_Bogomolov_inequality}
Let $\sX=\sqrt[r]{(X,D)}$ be the $r$-th root stack corresponding to the pair $(X,D)$, and let $E$ be a strongly  slope semistable torsion free coherent sheaf on 
$\sX$ with respect to the polarization $(\sO_{\sX}, \sO_X(1))$. Then we have
$$\Delta(E)=2\rk(E) c_2(E)-(\rk(E)-1)c_1(E)^2\geq 0.$$
\end{prop}

\subsubsection{Parabolic sheaves}

\begin{defn}(\cite{MY})\label{defn_parabolic_sheaves}
Let $E$ be a torsion-free coherent sheaf on $X$.  A {\em parabolic structure} on $E$ is given by a length $d$-filtration 
$$E=F_1(E)\supset F_2(E)\supset \cdots\supset F_r(E)\supset F_{r+1}(E)=E(-D),$$
together with a system of weights 
$$0\leq \alpha_1, \alpha_1,\cdots, \alpha_{r}<1.$$
We call $E_{\bullet}=(E,F_i(E))$ a rational parabolic sheaf associated with the divisor $D$ if all the weights  $\alpha_1, \alpha_1,\cdots, \alpha_{r}$ are all rational. 
Let $G_i(E)=F_i(E)/F_{i+1}(E)$. The Hilbert polynomial $\chi(G_i(E)(m))$ is called the $i$-th multiplicity polynomial of the weight $\alpha_i$. 
\end{defn}

As in \cite[Definition 1.8]{MY},  the parabolic Euler characteristic $\pa-\chi(E_\bullet)$ of $E_{\bullet}$ is defined as:
\begin{equation}\label{eqn_parabolic_Euler}
\chi(E(-D))+\sum_{i=0}^{d-1}\alpha_i \chi(G_i).
\end{equation}
The polynomial  $\pa-\chi(E_\bullet(m))$ is called the parabolic Hilbert polynomial of $E_\bullet$ and the polynomial 
$\pa-\chi(E_\bullet(m))/\rk(E)$ is denoted by $\pa-p_{E_\bullet}(m)$.

\begin{defn}\label{defn_parabolic_semistability}
The  parabolic sheaf of  $E_\bullet$  is said to be parabolic Gieseker stable (resp. parabolic semistable) if for every parabolic subsheaf $F_\bullet$ of $E_\bullet$ with 
$$0<\rk(F)<\rk(E)$$
we have
$$\pa-p_{F_\bullet}(m)<\pa-p_{E_\bullet}(m), \quad (\text{resp.}  \pa-p_{F_\bullet}(m)\leq \pa-p_{E_\bullet}(m)).$$

The parabolic degree $\padeg(E_{\bullet})$ is defined by
$$\padeg(E_{\bullet})=\int_{0}^{1}\deg(E_{\alpha})d\alpha+\rk(E)\cdot \deg(D).$$
We set the parabolic slope as $\pa-\mu(E_{\bullet})=\frac{\padeg(E_{\bullet})}{\rk(E)}$. 
Then $E_\bullet$  is  parabolic slope stable (resp. parabolic semistable) if for every parabolic subsheaf $F_\bullet\subset E_\bullet$ with 
$$0<\rk(F)<\rk(E)$$
we have
$$\pa-\mu(F_\bullet)<\pa-\mu(E_\bullet),  \quad (\text{resp.} \pa-\mu(F_\bullet)\leq \pa-\mu(E_\bullet).$$
\end{defn}

\subsubsection{Equivalence of categories}\label{subsubsec_equivalence_categories}
For the root stack 
$\sX=\sqrt[r]{(X,D)}$, we choose the generating sheaf  
$\Xi=\oplus_{i=0}^{r}\sO_{\sX}(\sD^{\frac{i}{r}})$.
Let $\Coh(\sX)$ be the abelian category of coherent sheaves on $\sX$, and $\Par_{\frac{1}{r}}(X,D)$ the abelian category of rational parabolic sheaves on $(X,D)$ with length $r$. 
There exist  two functors:
$$\sF_{\sX}: \Coh(\sX)\to \Par_{\frac{1}{r}}(X,D); \quad E\mapsto \sF_{\sX}(E)$$
where $\sF_{\sX}(E)_l=\pi_{*}(\E\otimes \sO_{\sX}(-l \sD))$; and    
$$G_{\sX}: \Par_{\frac{1}{r}}(X,D)\to \Coh(\sX); \quad E_{\bullet}\mapsto \int^{\zz}g_{\sX}(E_{\bullet})(l,l)$$
where $\int^{\zz}g_{\sX}(E_{\bullet})(l,l)$ is the colimit of wedges:
\[
\xymatrix{
g_{\sX}(E_{\bullet})(l,m)\ar[r]^{f_{d,m}}\ar[d]_{h_{l,m}}& g_{\sX}(E_{\bullet})(l,l)\ar[d]^{w(l)}\\
g_{\sX}(E_{\bullet})(m,m)\ar[r]^--{w(m)}& \sG
}
\]
where
\begin{enumerate}
\item $g_{\sX}(E_{\bullet}): \zz^{0}\times \zz\to \Coh(\sX)$ is a map given by:
$$(l,m)\mapsto \sO_{\sX}(l\sD)\otimes p^*E_m;$$
\item $m\geq l$ is an arrow in $\zz$, and the arrow $h_{l,m}$ is induced by the canonical section of the divisor, the arrow $f_{l,m}$ is induced by the filtration $p^*E_{\bullet}$, the arrow $w(l)$ is a dinatural transformation and $\sG$ is a sheaf in 
$\Coh(\sX)$. 
\end{enumerate}

We define that a parabolic sheaf $E_\bullet\in \Par_{\frac{1}{r}}(X,D)$ to be torsion free if $E_0$ is torsion free. 
Then we have:
\begin{thm}(\cite[Theorem 6.1]{BV})\label{thm_category_equivalence}
The functor $G_{\sX}$ maps torsion free sheaves on $X$ to torsion free sheaves on $\sX$. Moreover, $\sF_{\sX}$ and $G_{\sX}$ are inverse to each other when applied to torsion free sheaves. 
\end{thm}

\subsubsection{Parabolic Bogomolov inequality}\label{subsubsec_parabolic_inequality}

In this section we apply Proposition \ref{prop_Bogomolov_inequality} and Kawamata cover of $(X,D)$ constructed in \cite{Biswas} to deduce the parabolic Bogomolov inequality for parabolic semistable sheaves. 

Let $Y\to X$ be the Galois cover with Galois group $G=\Gal(\Rat(Y)/\Rat(X))$ as in \cite{Biswas}.
Let $r||G|$, then by base change of the root stack in the fibre product  diagram 
\[
\xymatrix{
\sY\ar[r]^{\tilde{f}}\ar[d]_{\pi_{Y}}& \sX\ar[d]^{\pi}\\
Y\ar[r]^{f}& X
}
\]
in Diagram \ref{eqn_key_diagram_Galois}, the Deligne-Mumford stack $\sY$ is the root stack of the pullback line bundle 
$(f^*\sO_X(D), f^*s_D)$.  Since the degree of $f^*\sO_X(D)$ is divided by $r$, the root stack associated this line bundle is trivial. Therefore 
we have $\sY\cong Y$ is a scheme and we have the following diagram:
\[
\xymatrix{
& \sX\ar[d]^{\pi}\\
Y\ar[ur]^{\tilde{f}}\ar[r]^{f}& X
}
\]
and the quotient stack $[Y/G]\cong \sX$. 

We are ready to state the parabolic Bogomolov inequality.  We set up some notations.
First let $W$ be a coherent sheaf on $\sX$. Then from Theorem \ref{thm_category_equivalence}, there exists a rational parabolic sheaf 
$(E, F_*, \alpha_{*})$ such that 
$F_i=\pi_{*}(W\otimes \sO_{\sX}(-i \sD))$ and $\alpha_i=\frac{i}{r}$.

Let $D=\sum_{\lambda=1}^{h} D_\lambda$ be the decomposition of $D$ into smooth irreducible components. 
Let $\widetilde{D}=(f^{*}D)_{\red}$ which is normal crossing. Then 
$f^*D_{\lambda}= k_{\lambda}r(f^*D_{\lambda})_{\red}$, where $  1 \leq \lambda\leq  h$, and $ k_{\lambda} \geq  1$  are
integers.
The torsion free coherent sheaf $W$ on $\sX=[Y/G]$ gives a $G$-equivariant torsion free coherent sheaf  on $Y$
which we still denote by $W$. 
Let 
$$\iota: \widetilde{D}\to Y$$
be the inclusion and $\iota^{\lambda}: \widetilde{D}_{\lambda}\to Y$ be the inclusion of the component $\widetilde{D}_{\lambda}$, 
where $\widetilde{D}_{\lambda}=(f^{*}D_{\lambda})_{\red}$.
Let 
$H_i=E|_{D}/F_i(E|_{D})$ be the sheaf on $D$. Then 
$$H_j=\bigoplus_{\lambda=1}^{h}\overline{\iota}_*^{\lambda}H_j^{\lambda},$$
where $\overline{\iota}^{\lambda}: D_{\lambda}\hookrightarrow D$ denotes the inclusion. 
Define 
$G_{i,\lambda}=H^{\lambda}_{i}/H^{\lambda}_{i+1}$
and 
$$\widetilde{G}_{i,\lambda}:=(f\iota^{\lambda})^*G_{i,\lambda}.$$

Then from \cite[Formula (3.15)]{Biswas} we have:
\begin{equation}
W=f^*E+\sum_{i=1}^{r}\sum_{\lambda=1}^{h}\sum_{j=1}^{k_{\lambda}m_i}\iota_*^{\lambda}\left(\widetilde{G}_{i,\lambda}\otimes N_{\widetilde{D}_{\lambda}}^{j}\right)
\end{equation}
in the K-theory $K_0(Y)$, where $N_{\widetilde{D}_{\lambda}}=\sO_{Y}(\widetilde{D}_{\lambda})|_{\widetilde{D}_{\lambda}}$ is the normal bundle to the divisor $\widetilde{D}_{\lambda}$. 

From equivalence between the categories $\Coh(\sX)$ and $\Par_{\frac{1}{r}}(X,D)$ in Theorem \ref{thm_category_equivalence}, the the main result in \cite{BV}, by choosing the generating sheaf $\Xi=\bigoplus_{i=0}^{r}\sO_{\sX}(i\sD)$, $W$ is modified semistable if and only if the corresponding rational parabolic sheaf $(E, F_*, \alpha_*)$ is parabolic semistable. But from \cite[Lemma 3.13]{Biswas}, the  rational parabolic sheaf $(E, F_*, \alpha_*)$ is parabolic semistable with respect to $\sO_X(1)$ if and only if the corresponding sheaf 
$W$ is orbifold semistable with respect to $f^*\sO_X(1)$. From Proposition \ref{prop_Bogomolov_inequality}, if $W$ is strongly semistable, then 
$$\Delta(W)\geq 0.$$
Thus we have
\begin{thm}\label{thm_parabolic_Bogomolov}(\cite{Biswas})
Let $W$ be an orbifold strongly  semistable torsion free coherent sheaf on the root stack $\sX$ such that its corresponding rational parabolic sheaf $(E, F_*, \alpha_*)$ is parabolic strongly semistable, then we have 
\begin{align*}
&c_2(E)+c_1(E)\cup\left(\sum_{i=1}^{r}\sum_{\lambda=1}^{h}\alpha_i\cdot r_{i\lambda}[D_{\lambda}]\right)
+\frac{1}{2}\left(\sum_{i=1}^{r}\sum_{\lambda=1}^{h}\alpha_i\cdot r_{i\lambda}[D_{\lambda}]\right)^2\\
&-\sum_{i=1}^{r}\sum_{\lambda=1}^{h}\left(\frac{\alpha_i^2\cdot r_{i\lambda}\cdot D_{\lambda}\cdot D_{\lambda}}{2}+\alpha_i\cdot d_{i\lambda}\right)\\
&\geq \frac{\rk(E)-1}{2\rk(E)}\left(c_1(E)+\sum_{i=1}^{r}\sum_{\lambda=1}^{h}\alpha_i\cdot r_{i\lambda}\cdot [D_{\lambda}]\right)^2
\end{align*}
where 
$r_{i\lambda}=\rk(G_{i,\lambda})$ and $d_{i\lambda}=\deg(G_{i,\lambda})$.
\end{thm}
\begin{proof}
The Chern class formula $c_2(W), c_1(W)$ are calculated in \cite[\S 4]{Biswas}.  Plug these into the inequality $\Delta(W)\geq 0$ we get the formula in the theorem. 
\end{proof}
\begin{rmk}
\cite{Biswas} proves that the orbifold semistability of $W$ is equivalent to the parabolic semistability of the corresponding parabolic sheaf $(E, F_*, \alpha_*)$, and the  parabolic semistability is equivalent to the modified semistability of $W$ again for the generating sheaf 
$\Xi=\oplus_{i=0}^{r}\sO_{\sX}(\sD^{\frac{i}{r}})$.
\end{rmk}


\appendix

\section{Higher dimension case}\label{subsec_higher_dimension_case}

In the appendix we generalize Langer's argument to higher dimension case.  For simplicity of the calculation of modified slopes, we restrict to a special case of smooth Deligne-Mumford stacks $\sX=[Z/G]$ which is a quotient stack such that the action of $G$ is diagonalizable.
We still let the generating sheaf $\Xi$ on $\sX$ satisfying Condition \ref{condition_star}.  Still let $d=\dim(\sX)$ be the dimension of $\sX$.

Let 
$F:\sX\to \sX$ be the absolute Frobenius map and it is identity in characteristic zero. Note that a coherent sheaf $E$ on $\sX$ is strongly modified slope semistable if and only if it is semistable under Frobenius pullbacks.   \S \ref{subsec_char_p_notations_Shepherd-Barron} implies that $L_{\Xi,\max}(E)$ and $L_{\Xi,\min}(E)$ are well defined rational numbers. 
Let $L\in \Pic(\sX)$ be a $\pi$-ample line bundle on $\sX$ such that 
$\bigoplus _{k=1}^{m-1}L^{\otimes k}$ is a generating sheaf. 
Recall that $\pi: \sX\to X$ is the map to its coarse moduli space. We choose a nef divisor $A$ on the scheme $X$ and let:
$$
\beta_{\rk}:=\beta_{\rk(E)}(A, H)=\left(\frac{\rk(E)\cdot (\rk(E)-1)(\max_{1\leq k\leq m-1}(\mu_{\Xi}(A\otimes L^{k})))}{p-1}\right)^{2}.$$

We  state several theorems generalizing Langer \cite[\S 3]{Langer}.

\begin{thm}\label{thm_Deligne-Mumford_1}
Let $D_1$ be a very ample divisor on $X$ and $\sD_1:=\pi^{-1}(D_1)$ and $\sD=\pi^{-1}(D)$ for $D$  a general element $D\in |D_1|$ .  If the restriction of a coherent sheaf $E$ on $\sX$ to $\sD$ 
 is not modified slope semistable  with respect to $H|_{D}$ and $\Xi|_{D}$, then  let $\mu_{\Xi,i}, r_i$ denote the modified slopes and ranks respectively in the Harder-Narasimhan filtration of $E|_{D}$, we have 
\begin{equation}\label{eqn_Deligne-Mumford-3-1}
\sum_{i<j}r_i r_j(\mu_{\Xi,i}-\mu_{\Xi,j})^2\leq H^{d}\Delta(E)+2\rk(E)^2(L_{\Xi,\max}(E)-\mu_{\Xi}(E), \mu_{\Xi}(E)-L_{\Xi,\min}(E)).
\end{equation}
\end{thm}

\begin{thm}\label{thm_Deligne-Mumford_2}
If a torsion free sheaf $E$ on $\sX$ is  strongly modified slope semistable,  we have 
$$\Delta(E)\cdot H^{d-2}\geq 0.$$
\end{thm}

\begin{thm}\label{thm_Deligne-Mumford_3}
If a torsion free sheaf $E$ on $\sX$ is just  modified slope semistable,  then we have 
$$H^{d}\cdot \Delta(E)\cdot H^{d-2}+\rk(\Xi)^2\beta_{\rk(E)}\geq 0.$$
\end{thm}

Before stating the last theorem, we introduce some notations.  First for torsion free  sheaves $G^\prime, G$ on $\sX$, we set 
$$\xi^{\Xi}_{G^\prime, G}=\frac{c_1(G^\prime)}{\rk(\Xi)\rk(G^\prime)}-\frac{c_1(G)}{\rk(\Xi)\rk(G)}.$$
We also set 
$$K^+:=\{D\in \text{Num}(X)| D^2H^{d-2}>0, DH^{d-1}\geq 0 \text{~for all nef~}H\}. $$
where $\text{Num}(X)=\Pic(X)\otimes\rr/\sim$ and $\sim$ is an equivalence relation meaning $L_1\sim L_2$ if and only if 
$L_1AH^{d-2}=L_2AH^{d-2}$ for all divisors $A$ on $X$.

\begin{thm}\label{thm_Deligne-Mumford_4}
If  we have 
$$H^{d}\cdot \Delta(E)\cdot H^{d-2}+\rk(\Xi)^2\beta_{\rk(E)}< 0,$$
then there exists a saturated subsheaf $E^\prime\subset E$ such that 
$\xi^{\Xi}_{E^\prime, E}\in K^{+}$.
\end{thm}

We prove these theorems by induction on the rank $\rk(E)$, and following Langer's method.  We only state the parts of the proof which are different to Langer's method in smooth case and refer to \cite[\S 3]{Langer} for detailed arguments in the proof which is the same as Langer. 
For the induction process, let $\Thm^{i}(\rk)$ represent  the statement that Theorem  4.i holds for ranks $\leq \rk$ for $i=1,2,3,4$ and 
$\Thm^{5}(\rk)$ represents that Theorem \ref{thm_Deligne-Mumford_2} holds for $\rk(E)\leq \rk$.

\subsection{$\Thm^{1}(\rk)$ implies $\Thm^{5}(\rk)$}

Suppose that the torsion free sheaf $E$ is strongly modified slope semistable with respect to $(H, \Xi)$, and 
$\Delta(E)\cdot H^{d-2}<0$.  We have $L_{\Xi,\max}(E)=L_{\Xi,\min}=\mu_{\Xi}(E)$.  Theorem $\Thm^{1}(\rk)$ implies that the restriction of $E$ to 
$H$ is still modified slope semistable.  Since $E$ is strongly modified slope semistable, $(F^k)^*E$ is also strongly modified slope semistable, and its restriction to a very general element in $|H|$ is strongly modified slope semistable. Therefore by induction the restriction  of $(F^k)^*E$  to a very general element in $H_1\cap\cdots\cap H_{d-1}$ for $H_1, \cdots, H_{d-1}\in |H|$ is  strongly modified slope semistable.   Therefore we are reduced to the two dimensional Deligne-Mumford stack case. Then this is Theorem \ref{thm_Deligne-Mumford_2-dim}.

\subsection{$\Thm^{5}(\rk)$ implies $\Thm^{3}(\rk)$}

First note that in this case, 
$$\beta_{\rk(E)}=\left(\frac{\rk(E)\cdot (\rk(E)-1)(\max_{1\leq k\leq m-1}(\mu_{\Xi}(A\otimes L^{k})))}{p-1}\right)^{2}.$$ 
Our polarization is $(H, \Xi)$, we first have the following inequality:
\begin{equation}\label{eqn_claim}
H^{d}\cdot \Delta(E)H^{d-2}+\rk(E)^2\rk(\Xi)^2(L_{\Xi,\max}(E)-\mu_{\Xi}(E))(\mu_{\Xi}(E)-L_{\Xi,\min}(E))\geq 0
\end{equation}
To prove this inequality, first from the finite property fdHPN in \S \ref{subsec_fdHN} there exists a positive integer $k$ such that all the quotients in the Harder-Narasimhan 
filtration of $(F^k)^*E$ are strongly modified slope semistable.  Consider the Harder-Narasimhan filtration 
$$0=E_0\subset E_1\subset\cdots\subset E_m=(F^k)^*E$$
and let $F_i=E_i/E_{i-1}$, $r_i=\rk(F_i)$, $\mu_i=\mu_{\Xi}(F_i)$. The Hodge index theorem (holds for smooth Deligne-Mumford stacks) implies that 
\begin{align*}
\frac{\Delta((F^k)^*E)H^{d-2}}{\rk(E)}&=\sum_{i}\frac{\Delta(F_i)H^{d-2}}{r_i}-\frac{1}{\rk(E)}\sum_{i<j}r_ir_j\left(\frac{c_1(F_i)}{r_i}-\frac{c_1(F_j)}{r_j}\right)^2H^{d-2}\\
&\geq \sum_{i}\frac{\Delta(F_i)H^{d-2}}{r_i}-\frac{\rk(\Xi)^2}{H^{d}\rk(E)}\sum_{i<j}r_ir_j(\mu_i-\mu_j)^2
\end{align*}
$\Thm^{5}(\rk)$ implies that $\Delta(F_i)H^{d-2}\geq 0$.  Therefore by \cite[Lemma 1.4]{Langer}, we have 
\begin{multline*}\frac{H^{d}\cdot \Delta(E)H^{d-2}}{\rk(E)}\geq \\
-\rk(E)\rk(\Xi)^2\left(\mu_{\Xi,\max}((F^k)^*E)-\mu_{\Xi}((F^k)^*E)\right)\left(\mu_{\Xi}((F^k)^*E)-\mu_{\Xi,\min}((F^k)^*E)\right)
\end{multline*}
Both sides are divided by $p^{2k}$, we get:
$$H^{d}\cdot \Delta(E)H^{d-2}+\rk(E)^2\rk(\Xi)^2(L_{\Xi,\max}(E)-\mu_{\Xi}(E))(\mu_{\Xi}(E)-L_{\Xi,\min}(E))\geq 0.$$

It is ready to prove  $\Thm^{3}(\rk)$.  Suppose that $E$ is just modified slope semistable.
We aim to use (\ref{eqn_claim}) and Corollary \ref{cor_alphaE}.  The method is the same as in \cite[\S 3.6]{Langer}, and we take 
an ample divisor $D$ on $X$ and set $H(t)=H+tD$.  Similar method shows that the Harder-Narasimhan filtration of $E$ with respect to 
$(H(t), \Xi)$ is independent of $t$ when $t$ is positively small.  Let $0=E_0\subset E_1\subset\cdots\subset E_m=E$ be the Harder-Narasimhan filtration with respect to $(H(t), \Xi)$.  We have (since $E$ is modified slope semistable)
$$\mu_{\Xi, H}(E)\geq \mu_{\Xi, H}(E_1)=\lim_{t\to 0}\mu_{\Xi, H(t)}(E_1)\geq \lim_{t\to 0}\mu_{\Xi, H(t)}(E)=\mu_{\Xi,H}(E).$$
Hence
$$\lim_{t\to 0}\mu_{\Xi,\max, H(t)}(E)=\lim_{t\to 0}\mu_{\Xi,\max, H(t)}(E_1)=\mu_{\Xi,H}(E),$$
and similarly, 
$$\lim_{t\to 0}\mu_{\Xi,\min, H(t)}(E)=\mu_{\Xi,H}(E).$$
Thus we can apply (\ref{eqn_claim}) and Corollary \ref{cor_alphaE}, we get the result
$$H^{d}\cdot \Delta(E)H^{d-2}+\rk(\Xi)^2\beta_{\rk(E)}\geq 0.$$

\subsection{$\Thm^{3}(\rk)$ implies $\Thm^{4}(\rk)$}

If we have the condition in $\Thm^{4}(\rk)$, i.e., $H^{d}\cdot \Delta(E)H^{d-2}+\rk(\Xi)^2\beta_{\rk(E)}< 0$. Then from $\Thm^{3}(\rk)$, $E$ is not modified slope semistable.  Let $E^\prime\subset E$ be the maximal destabilizing subsheaf of $E$ and set $E^{\prime\prime}=E/E^{\prime}$, 
$r^\prime=\rk(E^\prime)$, $r^{\prime\prime}=\rk(E^{\prime\prime})$.  First we calculate:
$$\frac{\Delta(E)H^{d-2}}{\rk(E)}
+\frac{r r^{\prime}\xi^2_{E^\prime, E}H^{d-2}}{r^{\prime\prime}}
=\frac{\Delta(E^\prime)H^{d-2}}{r^\prime}+\frac{\Delta(E^{\prime\prime})H^{d-2}}{r^{\prime\prime}}.$$
We also have $\rk(\Xi)^2\frac{\beta_{\rk(E)}}{\rk(E)}\geq \rk(\Xi)^2(\frac{\beta_{r^\prime}}{r^\prime}+\frac{\beta_{r^{\prime\prime}}}{r^{\prime\prime}})$. Since $H^{d}\cdot \Delta(E)H^{d-2}+\rk(\Xi)^2\beta_{\rk(E)}< 0$, and we require $H^d>0$, either $\xi^2_{E^\prime, E}>0$ or at least one of the 
$H^{d}\cdot\Delta(E^\prime)H^{d-2}+\rk(\Xi)^2\beta_{r^\prime}$ and $H^{d}\cdot\Delta(E^{\prime\prime})H^{d-2}+\rk(\Xi)^2\beta_{r^{\prime\prime}}$ is negative. 
Therefore the same argument as in \cite[Theorem 7.3.3]{HL} gives the result.

\subsection{$\Thm^{4}(\rk)$ implies $\Thm^{2}(\rk)$}

Suppose that $\Delta(E)H^{d-2}<0$.  The condition in $\Thm^{4}(\rk)$, applying to $(F^l)^*E$ (since $E$ is strongly modified slope semistable by the condition in $\Thm^{2}(\rk)$),  is:
$$H^{d}\cdot \Delta((F^l)^*E)H^{d-2}+\rk(\Xi)^2\beta_{\rk(E)}< 0$$
which is equivalent to 
$$l>\frac{1}{2}\log_{p}\left(-\frac{\rk(\Xi)^2\beta_{\rk(E)}}{H^{d}\cdot \Delta(E)H^{d-2}}\right).$$

Then for large $l$, there exists a saturated torsion free subsheaf $E^\prime\subset (F^l)^*E$ such that 
$\xi_{E^\prime, (F^l)^*E}\in K^+$. By ``self-duality" property of $K^+$ we have
$\xi_{E^\prime, (F^l)^*E}H^{d-1}>0$, which means that the sheaf $E$ is not strongly modified semistable, a contradiction. 

\subsection{$\Thm^{2}(\rk-1)$ implies $\Thm^{1}(\rk)$}

We use $\Pi=|H|$ to denote the complete linear system, and let $Z:=\{(D,x)\in \Pi\times X: x\in D\}$ be the incidence variety. 
Let 
$$p: Z\to \Pi; \quad  q:  Z\to X$$
be the corresponding projections.  For each $s\in \Pi$, let $Z_s$ be the scheme theoretic fiber of $p$ over the point $s$. 
Consider the following cartisian diagram:
\[
\xymatrix{
&\sZ\ar[dl]_{p\circ\pi}\ar[r]^{q}\ar[d]_{\pi}& \sX\ar[d]^{\pi}\\
\Pi&\ar[l]_{p} Z\ar[r]^{q}& X
}
\]
Then $\sZ$ is a Deligne-Mumford stack which is given by $\{(\pi^{-1}(D),x): (D,x)\in \Pi\times X,  x\in D\}$. 
The generating sheaf $\Xi$ on $\sX$ is  pullback under $q$ and gives a generating sheaf $q^*\Xi$ on $\sZ$ which is relative to $\Pi$. 

We work on the sheaf $q^*E$ for a torsion free sheaf $E$ on $\sX$, and let 
$$0=E_0\subset E_1\subset\cdots\subset E_m=q^*E$$
be the relative Harder-Narasimhan filtration with respect to $p\circ\pi$.
This means that there exists an open subset $U\subset \Pi$ such that all $F_i=E_i/E_{i-1}$ are flat over $U$ and for each $s\in U$ the fibers 
$(E_{\bullet})_s$ is the  Harder-Narasimhan filtration of $E_s=q^*E|_{\sZ_s}$ for $\sZ_s=\pi^{-1}(Z_s)$.  
From the proof of in \cite[\S 3.9]{Langer}, the relative Harder-Narasimhan filtration
is actually the Harder-Narasimhan filtration of $q^*E$ with respect to 
$$(p^*\sO_{\Pi}(1)^{\dim(\Pi)}q^*H, q^*\Xi).$$

By the finite property in \S \ref{subsec_fdHN}, for the sheaf $q^*E$,  there exists a positive integer $k$ such that all the quotients in the Harder-Narasimhan filtration of $(F^k)^*(q^*E)=q^*((F^k)^*E)$ are strongly modified semistable. 
We will prove  the inequality  (\ref{eqn_Deligne-Mumford-3-1}), and from \cite[Lemma 1.5]{Langer}, when applying to the polygons of the Harder-Narasimhan filtration for modified slopes, we just prove the case that all the graded pieces $F_i$'s are strongly modified slope semistable with respect to 
$(p^*\sO_{\Pi}(1)^{\dim(\Pi)}q^*H, q^*\Xi)$.

We perform the same argument as in \cite[\S 3.9]{Langer}, and let $\Lambda\subset \Pi$ be a pencil. Set
$Y=p^{-1}(\Lambda)$, and $\sY=(p\circ\pi)^{-1}(\Lambda)\subset \sZ$. 
Since $q|_{Y}: Y\to X$ is the blow up of $X$ along the base locus $B$ of $\Lambda$, we can view 
$q|_{\sY}: \sY\to \sX$ to be  the stacky blow up of $\sX$ along the  locus $\sB=\pi^{-1}(B)$.
If $d\geq 3$, then $B$ is a smooth connected variety. So $\sB$ is a smooth connected substack, and there is only one exceptional divisor 
$\sN$ for $q|_{\sY}$.  We write down 
$$c_1(F_i|_{Y})=q|_{\sY}^*\sM_i+ b_i \sN$$
where $\sM_i$ are divisors on $\sX$ which are pullbacks of divisors $M_i$ on $X$  and $b_i$ are rational numbers. 
If the dimension $d=2$, then $B$ consists of $N=H^d$ distinct points and $\sB$ consists of $N$ distinct  stacky points.  Let 
$\sN_1,\cdots, \sN_{N}$ be the exceptional divisors of $q|_{\sY}$.  There exist rational numbers $b_{ij}$ and divisors  $\sM_i$ such that 
$$c_1(F_i|_{Y})=q|_{\sY}^*\sM_i+ \sum_{j}b_{ij} \sN_j.$$
Let $b_i=(\sum_{j}b_{ij})/N$. We have
$$\mu_{\Xi,i}=\frac{c_1(F_i|_{Y})p^*\sO_{\Pi}(1)q^*H^{d-2}}{r_i\rk(\Xi)}=\frac{\sM_i\cdot H^{d-1}+b_iN}{r_i\rk(\Xi)}.$$

$\Thm^{2}(\rk-1)$ implies that $\Delta(F_j|_{Y})p^*\sO_{\Pi}(1)q^*H^{d-2}\geq 0$  for every $j$. We calculate
\begin{align*}
\frac{N\Delta(E)q|_{\sY}^*H^{d-2}}{\rk(E)}&=\sum_{i}\frac{N\Delta(F_i|_{\sY})q|_{\sY}^*H^{d-2}}{r_i}
-\frac{N}{\rk(E)}\sum_{i<j}r_ir_j\left(\frac{c_1(F_i|_{\sY})}{r_i}-\frac{c_1(F_j|_{\sY})}{r_j}\right)^2\cdot q|_{\sY}^*H^{d-2}\\
&\geq \frac{N}{\rk(E)}\sum_{i<j}r_ir_j\left(N\left(\frac{b_i}{r_i}-\frac{b_j}{r_j}\right)^2-\left(\frac{\sM_i}{r_i}-\frac{\sM_j}{r_j}\right)^2H^{d-2}\right)\\
&\geq \frac{1}{\rk(E)}\sum_{i<j}r_ir_j\left((N)^2\left(\frac{b_i}{r_i}-\frac{b_j}{r_j}\right)^2-\left(\frac{\sM_iH^{d-1}}{r_i}-\frac{\sM_jH^{d-1}}{r_j}\right)^2\right).
\end{align*}
The last inequality is from Hodge index theorem for smooth Deligne-Mumford stacks, and from the slope $\mu_{\Xi,i}$, the last expression above  gives 
$$2\sum_{i}N b_i\mu_{\Xi, i}-\frac{1}{\rk(E)}\sum_{i<j}r_i r_j (\mu_{\Xi,i}-\mu_{\Xi,j})^2.$$

To prove the claim, first $(q|_{\sY})_*(E_i|_{\sY})\subset E$  implies that 
$$\frac{\sum_{j\leq i}\sM_jH^{d-1}}{\rk(\Xi)\sum_{j\leq i}r_j}\leq \mu_{\Xi, \max}(E)$$
which gives the inequality:
\begin{equation}
\sum_{j\leq i}b_j N\geq \sum_{j\leq i}\rk(\Xi)r_j(\mu_{\Xi, j}-\mu_{\Xi, \max}(E))
\end{equation}
Therefore
\begin{align*}
\sum_{i}Nb_i \mu_{\Xi,i}&=\sum_{i}\left(\sum_{i<j}Nb_j\right)(\mu_{\Xi,i}-\mu_{\Xi,i+1})\\
&\geq 
\sum_{i}\left( \sum_{j\leq i}\rk(\Xi)r_j(\mu_{\Xi, j}-\mu_{\Xi, \max}(E))\right)(\mu_{\Xi,i}-\mu_{\Xi,i+1})\\
&=\rk(\Xi)\cdot \sum_{i<j}\frac{r_ir_j}{\rk(E)}(\mu_{\Xi,i}-\mu_{\Xi,j})^2+\rk(E)(\mu_{\Xi}(E)-\mu_{\Xi,\max}(E))(\mu_{\Xi}(E)-\mu_{\Xi,\min}(E)).
\end{align*}
So we get:
$$\frac{N\Delta(E)q|_{\sY}^*H^{d-2}}{\rk(E)}\geq 
 \sum_{i<j}\frac{2\rk(\Xi)-1}{\rk(E)}r_i r_j(\mu_{\Xi,i}-\mu_{\Xi,j})^2+2\rk(E)(\mu_{\Xi}(E)-\mu_{\Xi,\max}(E))(\mu_{\Xi}(E)-\mu_{\Xi,\min}(E)).$$
 
 \subsection{Restriction  theorems}

We generalize the restriction theorem of Langer to  smooth Deligne-Mumford stacks in higher dimensions.
We give a general statement for the (\ref{eqn_claim}).
Recall from Corollary \ref{cor_alphaE}, 
$$\alpha_{\Xi}(E):=\max(L_{\Xi,\max}(E)-\mu_{\Xi, \max}(E), \mu_{\Xi, \min}(E)-L_{\Xi,\min}(E)).$$
Let $L\in \Pic(\sX)$ be a $\pi$-ample line bundle on $\sX$ such that there exists $m\in\zz_{>0}$ making 
$\bigoplus_{i=1}^{m-1}L^{\otimes i}$ a generating sheaf.    
Let $A$ be a nef divisor for $X$,  such that $\pi_{*}(T_{X}\otimes\sum_{k=1}^{m-1}(L^{k}))(A)$ is globally generated.
Then we have
$$\alpha_{\Xi}(E)\leq \frac{\rk(E)-1}{p-1} (\max_{1\leq k\leq m-1}\{\mu_{\Xi}(A\otimes (L^{k}))\}).$$

\begin{thm}\label{thm_general_inequality}
Let $E$ be a torsion free sheaf on a smooth Deligne-Mumford stack $\sX$. Then we have 
\begin{equation}\label{eqn_claim1}
H^{d}\cdot \Delta(E)H^{d-2}+\rk(E)^2\rk(\Xi)^2(L_{\Xi,\max}(E)-\mu_{\Xi}(E))(\mu_{\Xi}(E)-L_{\Xi,\min}(E))\geq 0
\end{equation}
and 
\begin{equation}\label{eqn_claim2}
H^{d}\cdot \Delta(E)H^{d-2}+\rk(E)^2\rk(\Xi)^2(\mu_{\Xi,\max}(E)-\mu_{\Xi}(E))(\mu_{\Xi}(E)-\mu_{\Xi,\min}(E))\geq 0
\end{equation}
\end{thm}
\begin{proof}
The proof of (\ref{eqn_claim1}) is the same as in Claim (\ref{eqn_claim}). The proof of formula  (\ref{eqn_claim2}) is the same as 
\cite[Theorem 5.1]{Langer}.
\end{proof}

\begin{thm}\label{thm_restriction_theorem}
Let $E$ be a torsion free sheaf of rank $\rk(E)\geq 2$ on a smooth Deligne-Mumford stack $\sX$.  Suppose that $E$ is slope modified  stable with respect to $(\Xi, \sO_{X}(1)=H)$. Let 
$D\subset |mH|$  be a normal divisor  such that $E|_{\sD}$ has no torsion where $\sD=\pi^{-1}(D))$.  If 
$$m>\lfloor\frac{\rk(E)-1}{\rk(E)}\Delta(E)H^{d-2}+\frac{1}{H^d \rk(E)(\rk(E)-1)}+\frac{(\rk(E)-1)\beta_{\rk}}{H^d \rk(E)} \rfloor$$
then $E|_{\sD}$ is slope modified stable with respect to $(\Xi|_{\sD}, H|_{\sD})$.
\end{thm}
\begin{proof}
The proof is the same as 
\cite[Theorem 5.2]{Langer}.
\end{proof}

\section{Bogomolov's inequality for Higgs sheaves}\Label{appendix_B}

\begin{center}
By Yunfeng Jiang, Promit Kundu and Hao Max Sun
\end{center}

\vspace{5 mm}

Let $k$ be an algebraically closed field of characteristic $p\geq0$
and $\mathcal{X}$ be a smooth tame projective Deligne-Mumford stack
of dimension $d$ over $k$ with coarse moduli space
$\pi:\mathcal{X}\rightarrow X$. Let $\Xi$ be a generating sheaf on
$\mathcal{X}$ satisfying Condition $\star$ in \ref{condition_star}, 
and let $H$ an ample divisor on $X$. 

\begin{defn}
A Higgs sheaf $(E, \theta)$ is a pair consisting of a coherent sheaf
$E\in \Coh(\mathcal{X})$ and an
$\mathcal{O}_{\mathcal{X}}$-homomorphism $\theta: E\rightarrow
E\otimes\Omega_{\mathcal{X}}$ satisfying the integrability condition
$\theta\wedge\theta=0$. We say a Higgs sheaf $(E, \theta)$ a system
of Hodge sheaves if there is a decomposition $E=\oplus E^i$ such
that $\theta: E^i\rightarrow
E^{i-1}\otimes_{\mathcal{O}_{\mathcal{X}}}\Omega_{\mathcal{X}}$
\begin{enumerate}
\item  We say that $(E, \theta)$ is slope
semistable if $\mu_{\Xi}(E^\prime)\leq\mu_{\Xi}(E)$ for every Higgs
subsheaf $(E^\prime,\theta^\prime)$ of $(E,\theta)$.

\item A system of Hodge sheaves $(E, \theta)$ is slope semistable if the
inequality $\mu_{\Xi}(E^\prime)\leq\mu_{\Xi}(E)$ is satisfied for
every subsystem of Hodge sheaves $(E^\prime,\theta^\prime)$ of
$(E,\theta)$.
\end{enumerate}
\end{defn}

We recall the main results of Ogus and Vologodsky \cite{OV}, where
the theory is for schemes, but in \'etale topology it works for
Deligne-Mumford stacks.

Assume that $p>0$. Let $S$ be a scheme over $k$ and
$f:\mathcal{X}\rightarrow S$ be a morphism of stacks over $k$. A
lifting of $\mathcal{X}/S$ modulo $p^2$ is a morphism
$\widetilde{f}: \widetilde{\mathcal{X}}\rightarrow \widetilde{S}$ of
flat $\mathbb{Z}/p^2\mathbb{Z}$-stacks such that $f$ is the base
change of $\widetilde{f}$ by the closed embedding $S\rightarrow
\widetilde{S}$ defined by $p$. Let $\mbox{MIC}_{p-1}(\mathcal{X}/S)$
be the category of $\mathcal{O}_{\mathcal{X}}$-modules with an
integrable connection whose $p$-curvature is nilpotent of level
$\leq p-1$. Let $\mbox{HIG}_{p-1}(\mathcal{X}^{(1)}/S)$ denote the
category of Higgs $\mathcal{O}_{\mathcal{X}^{(1)}}$-modules with a
nilpotent Higgs field of level $\leq p-1$. We have the following
theorem of Ogus and Vologodsky (\cite[Theorem 2.8]{OV})

\begin{thm}
If $f:\mathcal{X}\rightarrow S$ is a smooth morphism with a lifting
$\widetilde{\mathcal{X}}^{(1)}\rightarrow \widetilde{S}$ of
$\mathcal{X}^{(1)}\rightarrow S$ modulo $p^2$, then the Cartier
transform
$$C_{\mathcal{X}/S}:\mbox{MIC}_{p-1}(\mathcal{X}/S)\rightarrow\mbox{HIG}_{p-1}(\mathcal{X}^{(1)}/S)$$
defines an equivalence of categories with quasi-inverse
$$C^{-1}_{\mathcal{X}/S}:\mbox{HIG}_{p-1}(\mathcal{X}^{(1)}/S)\rightarrow\mbox{MIC}_{p-1}(\mathcal{X}/S).$$
\end{thm}

\begin{lem}\label{K}
Let $(E,\theta)\in\mbox{HIG}_{p-1}(\mathcal{X}^{(1)}/S)$. Then we
have $[C^{-1}_{\mathcal{X}/S}(E)]=F_g^*[E]$, where $[\cdot]$ denotes
the class of a coherent sheaf in the Grothendieck group
$K_0(\mathcal{X})$.
\end{lem}
\begin{proof}
See \cite[Lemma 2]{Langer2}.
\end{proof}

\begin{cor}\label{cor}
Assume $S=\spec k$, and let
$(E,\theta)\in\mbox{HIG}_{p-1}(\mathcal{X}^{(1)}/S)$. Then $(E,
\theta)$ is slope semistable with respect to $(H,\Xi)$ iff
$C^{-1}_{\mathcal{X}/S}(E)$ is slope $\nabla$-semistable with
respect to $(F_g^*H,F_g^*\Xi)$.
\begin{proof}
The proof is the same as that of \cite[Corollary 1]{Langer2}.
\end{proof}
\end{cor}

\begin{lem}\label{def}
Let $(E, \theta)$ be a torsion free slope semistable Higgs sheaf on
$\mathcal{X}$. Then there exists an $\mathbb{A}^1$-flat family of
Higgs sheaves $(\widetilde{E}, \widetilde{\theta})$ on
$\mathcal{X}\times\mathbb{A}^1$ such that the restriction
$(\widetilde{E}_t,\widetilde{\theta}_t)$ to the fiber over any
closed point $t\in \mathbb{A}^1$ is isomorphic to $(E,\theta)$ and
$(E_0, \theta_0)$ is a slope semistable system of Hodge sheaves.
\end{lem}
\begin{proof}
See \cite[Corollary 5.7]{Langer3}.
\end{proof}

\begin{prop}\label{pro}
Let $V$ be a torsion free sheaf on $\mathcal{X}$, then $$H^{d-2}\Delta(V)
\geq-\frac{(\rk V\rk \Xi)^2}{H^d}(L_{\Xi,\max}(V)-\mu_{\Xi}(V))(\mu_{\Xi}(V)-L_{\Xi,\min}(V))$$

\end{prop}
\begin{proof}
The proposition follows from Theorem \ref{thm_Deligne-Mumford_2}  by the same arguments as in the
proof of Theorem \ref{thm_Deligne-Mumford_3}.
\end{proof}

\begin{thm}
Assume $p=0$. Let $(E, \theta)$ be a slope semistable Higgs sheaf
with respect to $(H,\Xi)$. Then we have $H^{d-2}\Delta(E)\geq0$.
\end{thm}
\begin{proof}
Deforming $(E,\theta)$ to a system of Hodge sheaves (see Lemma
\ref{def}) we can assume that $(E, \theta)$ is nilpotent. Now we use
the standard reduction to positive characteristic technique, which
we recall for the convenience of the reader (see \cite[section
2]{MO}). There exists a finitely generated $\mathbb{Z}$-algebra
$R\subset k$ and a tame smooth Deligne-Mumford stack
$\widetilde{\mathcal{X}}\rightarrow S=\spec R$ such that
$\mathcal{X}=\widetilde{\mathcal{X}}\times_S\spec k$. Let
$\widetilde{\pi}: \widetilde{\mathcal{X}}\rightarrow \widetilde{X}$
be its coarse moduli space. We can assume that
$X=\widetilde{X}\times_S\spec k$, $\pi$ is induced by
$\widetilde{\pi}$ after base change, and there exists an ample
divisor $\widetilde{H}$ on $\widetilde{X}$ extending $H$ and a
generating sheaf $\widetilde{\Xi}$ on $\widetilde{\mathcal{X}}$
extending $\Xi$ such that its restriction to every component in
$I\widetilde{\mathcal{X}}_1$ is a direct sum of locally free
coherent sheaves of the same rank. We can also assume that there
exists an $S$-flat family of Higgs sheaves $(\widetilde{E},
\widetilde{\theta})$ on $\widetilde{\mathcal{X}}$ extending $(E,
\theta)$.

Shrinking $S$, by openness of semistability we can assume that
$(\widetilde{E}_s, \widetilde{\theta}_s)$ is slope semistable with
respect to $(\widetilde{\Xi}_s, \widetilde{H}_s)$ for any $s\in S$.
Choose a closed point $s\in S$ such that the characteristic $q$ of
the residue field $k(s)$ is $\geq \rk E$. Then the stack
$\widetilde{\mathcal{X}}\times_S\spec(R/m_s^2)$ is a lifting of
$\widetilde{\mathcal{X}}_s$ modulo $q^2$. By Corollary \ref{cor},
one can associate to $(\widetilde{E}_s, \widetilde{\theta}_s)$ a
slope $\nabla$-semistable sheaf with integrable connection $(V_s,
\nabla_s)$ with respect to
$(F_g^*\widetilde{H}_s,F_g^*\widetilde{\Xi}_s)$. From Lemma \ref{K},
it follows that
\begin{equation}
\widetilde{H}_s^{d-2}\Delta(V_s)=q^2\widetilde{H}_s^{d-2}\Delta(\widetilde{E}_s).
\end{equation}

Let $0=V_0\subset V_1\subset\cdots\subset V_m=V_s$ be the usual
Harder-Narasimhan filtration of $V_s$, then by \cite[Lemma 2.7]{JK},
the induced morphisms $V_i\rightarrow
(V_s/V_i)\otimes\Omega_{\widetilde{\mathcal{X}}_s}$ are nonzero
$O_{\widetilde{\mathcal{X}}_s}$-morphisms. Take a nef divisor $A$ on
$\widetilde{X}_s$ such that $\pi_*(T_{\widetilde{\mathcal{X}}_s}\otimes\widetilde{\Xi}_s)(A)$ is globally
generated. From \cite[Proposition 2.10 and Corollary 2.11]{JK}, it follows that
$$\mu_{\widetilde{\Xi}_s, \max}(V_s)-\mu_{\widetilde{\Xi}_s, \min}(V_s)\leq(\rk V_s-1)\left(\frac{\widetilde{H}_s^{d-1}A}{\widetilde{H}_s^d}+M\right)$$ and
$$\max(L_{\widetilde{\Xi}_s, \max}(V_s)-\mu_{\widetilde{\Xi}_s, \max}(V_s), \mu_{\widetilde{\Xi}_s,\min}(V_s)-L_{\widetilde{\Xi}_s, \min}(V_s))
\leq\frac{\rk V_s-1}{q-1}\left(\frac{\widetilde{H}_s^{d-1}A}{\widetilde{H}_s^d}+M\right),$$ for some positive constant $M$. They imply that
$$(L_{\widetilde{\Xi}_s,\max}(V_s)-L_{\widetilde{\Xi}_s,\min}(V_s))\leq\frac{\rk V_s-1}{1-\frac{1}{q}}\left(\frac{\widetilde{H}_s^{d-1}A}{\widetilde{H}_s^d}+M\right).$$
Hence Proposition \ref{pro} gives
$$q^2\widetilde{H}_s^{d-2}\Delta(\widetilde{E}_s)=\widetilde{H}_s^{d-2}\Delta(V_s)
\geq-\frac{(\rk V_s\rk \Xi)^2}{\widetilde{H}_s^d}(\frac{\rk V_s-1}{1-\frac{1}{q}})^2\left(\frac{\widetilde{H}_s^{d-1}A}{\widetilde{H}_s^d}+M\right)^2.$$
Taking sufficiently large $q$, one obtains $$H^{d-2}\Delta(E)=\widetilde{H}_s^{d-2}\Delta(\widetilde{E}_s)\geq0.$$
\end{proof}

Using the same argument in \cite[section 6]{Langer2}, one can recover \cite[Theorem 1.1]{CT}:
\begin{thm}
Assume that $p=0$, $d=2$, and the canonical line bundle $K_{\mathcal{X}}$ is nef, then we have
$$c^2_1(T_{\mathcal{X}})\leq3c_2(T_{\mathcal{X}}).$$
\end{thm}



\subsection*{}

\end{document}